\documentclass[11pt,a4paper]{amsart}
\usepackage{amscd,amssymb,amsmath,latexsym,graphicx,epsfig,color,url}
\usepackage[active]{srcltx}
\usepackage[all]{xy}
\usepackage{hyperref}
\hypersetup{breaklinks=true}

\def\ov#1{{\overline{#1}}}

\def\wt#1{{\widetilde{#1}}}

\newcommand{\MV}{\operatorname{MV}}
\newcommand{\conv}{\operatorname{conv}}
\newcommand{\newton}{\operatorname{N}}
\newcommand{\FS}{\operatorname{FS}}

\newcommand{\supp}{\operatorname{supp}}
\newcommand{\codim}{\operatorname{codim}}
\newcommand{\vol}{\operatorname{vol}}

\newcommand{\Res}{\operatorname{Res}}
\newcommand{\Deltaa}{\Delta_{\rm ang}}
\newcommand{\Deltar}{\Delta_{\rm rad}}
\newcommand{\mult}{\operatorname{mult}}

\newcommand{\SL}{\operatorname{SL}}
\newcommand{\GL}{\operatorname{GL}}
\newcommand{\NJ}{\operatorname{NJ}}

\renewcommand{\d}{{\rm d}}
\newcommand{\e}{\operatorname{e}}
\newcommand{\dist}{\operatorname{dist}}
\renewcommand{\i}{{\rm i}}

\newcommand{\Prob}{\operatorname{Prob}}
\newcommand{\Elim}{\operatorname{Elim}}
\newcommand{\h}{\operatorname{h}}
\newcommand{\haar}{\text{\rm Haar}}

\newcommand{\C}{\mathbb{C}}

\newcommand{\N}{\mathbb{N}}
\renewcommand{\P}{\mathbb{P}}

\newcommand{\R}{\mathbb{R}}

\newcommand{\Z}{\mathbb{Z}}
\newcommand{\E}{\mathbb{E}}

\newcommand{\cA}{{\mathcal A}}
\newcommand{\cB}{{\mathcal B}}

\newcommand{\bfa}{{\boldsymbol{a}}}
\newcommand{\bfb}{{\boldsymbol{b}}}

\newcommand{\bfe}{{\boldsymbol{e}}}
\newcommand{\bff}{{\boldsymbol{f}}}
\newcommand{\bfg}{{\boldsymbol{g}}}

\newcommand{\bfu}{{\boldsymbol{u}}}
\newcommand{\bfv}{{\boldsymbol{v}}}
\newcommand{\bfw}{{\boldsymbol{w}}}
\newcommand{\bfx}{{\boldsymbol{x}}}
\newcommand{\bfy}{{\boldsymbol{y}}}
\newcommand{\bfz}{{\boldsymbol{z}}}

\newcommand{\bfQ}{{\boldsymbol{Q}}}

\newcommand{\bfalpha}{{\boldsymbol{\alpha}}}
\newcommand{\bfbeta}{{\boldsymbol{\beta}}}

\newcommand{\bfnu}{{\boldsymbol{\nu}}}
\newcommand{\bfxi}{{\boldsymbol{\xi}}}
\newcommand{\bfchi}{{\boldsymbol{\chi}}}
\newcommand{\bfzero}{{\boldsymbol{0}}}

\newcommand{\bfcA}{{\boldsymbol{\mathcal A}}}

\numberwithin{equation}{section}

\theoremstyle{definition}
\newtheorem{definition}[equation]{Definition}

\newtheorem{remark}[equation]{Remark}
\newtheorem{example}[equation]{Example}

\theoremstyle{plain}
\newtheorem{lemma}[equation]{Lemma}
\newtheorem{proposition}[equation]{Proposition}
\newtheorem{theorem}[equation]{Theorem}
\newtheorem{corollary}[equation]{Corollary}

\begin{document}

\title[Equidistribution of the solutions of sparse polynomial systems]{Quantitative equidistribution
  for the solutions of systems of sparse polynomial equations}
\author[D'Andrea]{Carlos D'Andrea}
\address{D'Andrea: Universitat de Barcelona,
Departament d'{\`A}lgebra i Geometria.
Gran Via~585, 08007 Barcelona, Spain}
\email{cdandrea@ub.edu}
\urladdr{\url{http://atlas.mat.ub.es/personals/dandrea/}}

\author[Galligo]{Andr\'e Galligo}
\address{Galligo: Universit\'e de Nice-Sophia Antipolis,
Laboratoire de Math\'ematiques. Parc Valrose, 06108 Nice Cedex 02, France}
\email{galligo@unice.fr}
\urladdr{\url{http://math.unice.fr/~galligo/}}

\author[Sombra]{Mart{\'\i}n~Sombra}
\address{Sombra: ICREA and Universitat de Barcelona, Departament d'{\`A}lgebra i Geometria.
Gran Via~585, 08007 Barcelona, Spain}
\email{sombra@ub.edu}
\urladdr{\url{http://atlas.mat.ub.es/personals/sombra/}}

\date{\today} 
\thanks{D'Andrea was partially supported by the MICINN research
  project MTM2010-20279 (Spain).  Galligo was partially supported by the
European Marie Curie network SAGA. Sombra was partially supported by
  the MICINN research project MTM2009-14163-C02-01 and the MINECO research project
MTM2012-38122-C03-02 (Spain). }

\begin{abstract}
  For a system of Laurent polynomials $f_{1},\dots, f_{n}\in
  \C[x_{1}^{\pm1},\dots, x_{n}^{\pm1}]$ whose coefficients are not too
  big with respect to its directional resultants, we show that the solutions
  in the algebraic torus $(\C^{\times})^{n}$ of the system of
  equations $f_{1}=\dots=f_{n}=0$, are approximately equidistributed
  near the unit polycircle. This generalizes to the
  multivariate case a classical result due to Erd\"os and
  Tur\'an on the distribution of the arguments of the roots of a
  univariate polynomial.

  We apply this result to bound the number of real roots
  of a system of Laurent polynomials, and to study the
  asymptotic distribution of the roots of systems of Laurent
  polynomials over $\Z$ and of random systems of Laurent
  polynomials  over~$\C$.
\end{abstract}
\maketitle



\vspace{-7mm}

\section{Introduction and statement of results}

A celebrated result due to Erd\"os and Tur\'an says that, for a
univariate polynomial over $\C$ whose middle coefficients are not too
big with respect to its extremal coefficients, the arguments of its roots are
approximately equidistributed~\cite{ET50}.  Combined with a recent
result of Hughes and Nikeghbali~\cite{HN08}, this shows that the roots
of such a polynomial cluster near the unit circle.

We introduce some notation to make this result precise. Let $Z$ be an effective 
cycle of~$\C^{\times}=\C\setminus \{0\}$ of dimension 0, that is, a formal finite sum
\begin{displaymath}
  Z=\sum_{\xi} m_{\xi}[\xi] 
\end{displaymath}
with $\xi\in \C^{\times}$ and $m_{\xi}\in \N$ with $m_\xi=0$ for all
but finitely many $\xi$, as in \cite[\S1.3]{Fulton:IT}.  The degree of
$Z$, denoted $\deg(Z)$, is defined as the sum of its multiplicities
$m_{\xi}$. We assume that $Z\ne 0$ or, equivalently, that
$\deg(Z)\ge1$.

For each $ -\pi \le \alpha<\beta\le \pi$, consider the cycle
\begin{displaymath}
  Z_{\alpha,\beta}= \sum_{\alpha < \arg(\xi) \le \beta} m_{\xi}[\xi] ,
\end{displaymath}
where $\arg(\xi)$ denotes the argument of $\xi$.  The {\em angle
  discrepancy} of $Z$ is defined as
\begin{equation*}
  \Deltaa(Z)= \sup_{-\pi \le \alpha<\beta\le \pi} \bigg|  \frac{\deg(Z_{\alpha,\beta})}{\deg(Z)}- \frac{\beta-\alpha}{2\pi}\bigg|.
\end{equation*}
For example, when $Z$ is the zero set of $x^d-1$ in
$\mathbb{C}^\times$, we have that $\Deltaa(Z)=\frac{1}{d}$.

For $0<\varepsilon <1$, consider also the cycle 
\begin{displaymath}
  Z_{\varepsilon}= \sum_{1-\varepsilon <
  |\xi| < ({1-\varepsilon})^{-1}} m_{\xi}[\xi].
\end{displaymath}
The  {\em radius discrepancy} of $Z$ with
 respect to $\varepsilon$ is defined as
$$
\Deltar(Z, \varepsilon)=1- \frac{\deg(Z_{\varepsilon})}{\deg(Z)}.
$$
For example, when $Z$ is the zero set of
$x^d-1$ in $\mathbb{C}^\times$, we have that $\Deltar(Z,\varepsilon)=0$ for all $\varepsilon$.

For a polynomial $f\in \C[x]\setminus\{0\}$, we denote by $Z(f)$ the
0-dimensional effective cycle of $\C^{\times}$ defined by its roots
and their corresponding multiplicities.  We also set
$\|f\|_{\sup}=\sup_{|z|=1}|f(z)|$.

\begin{theorem}\label{mtunivariate}
Let  $f= a_0+\cdots+a_dx^d  \in \C[x] $
with $d\ge1$ and $a_0a_d \ne 0,$ and $0<\varepsilon <1$. Then 
\begin{equation*}
\Deltaa(Z(f)) \le c\,  \sqrt{\frac{1}{d}\log\bigg(\frac{\|f\|_{\sup}}{\sqrt{|a_0a_d|}}\bigg)},\quad
\Deltar(Z(f),\varepsilon)  \le \frac{2}{\varepsilon d}  \, \log\bigg(\frac{\|f\|_{\sup}}{\sqrt{|a_0a_d|}}\bigg),
\end{equation*}
with   $c=\sqrt{2\pi/G}= 2.5619\dots$, where
$G=\sum_{m=0}^\infty
\frac{(-1)^{m}}{(2m+1)^2}=0.915965594\dots$ is  Catalan's constant.
\end{theorem}

The more interesting (and hardest) part is the bound for the angle discrepancy.
The original Erd\"os-Tur\'an result states \cite{ET50}
$$
\Deltaa(Z(f)) \le 16\,  \sqrt{\frac{1}{d}\log\bigg(\frac{\sum_{j}|a_{j}|}{\sqrt{|a_0a_d|}}\bigg)}.
$$
A few years after that paper, Ganelius \cite{Gan54} replaced the
$\ell^{1}$-norm $\sum_{j}|a_{j}|$ by the smaller quantity $\Vert
f\Vert_{\sup}$ and improved the value of the constant to $c\le
\sqrt{2\pi/G}$.  On the other hand, Amoroso and Mignotte
\cite{AM96} showed that the optimal value of $c$ cannot be smaller
than $\sqrt{2}$. The bound for the radius discrepancy is due to Hughes
and Nikeghbali~\cite{HN08}.

\medskip Here, we study the distribution of the solutions of a system
of multivariate polynomial equations in the algebraic torus
$(\C^\times)^n$. For instance, consider the following
system of bivariate polynomials:
\begin{equation}\label{33}
  f_{1}=x_{1}^{13}+x_{1}x_{2}^{12}+x_{2}^{13}+1, \quad
  f_{2}=x_{1}^{12}x_{2}-x_{2}^{13}-x_{1}x_{2}+1 \quad \in \C[x_{1},x_{2}].
\end{equation}
These are polynomials with moderate degree and small integer
coefficients.  By direct computation, we can verify that the solutions
in $(\C^{\times})^{2}$ of the system of equations $f_{1}=f_{2}=0$ are
aproximately equidistributed near the unit polycircle $S^{1}\times S^1 $
(Figure~\ref{fig:6}).
\begin{figure}[htpb]
  \begin{tabular}{ccc}
\includegraphics[scale=0.30]{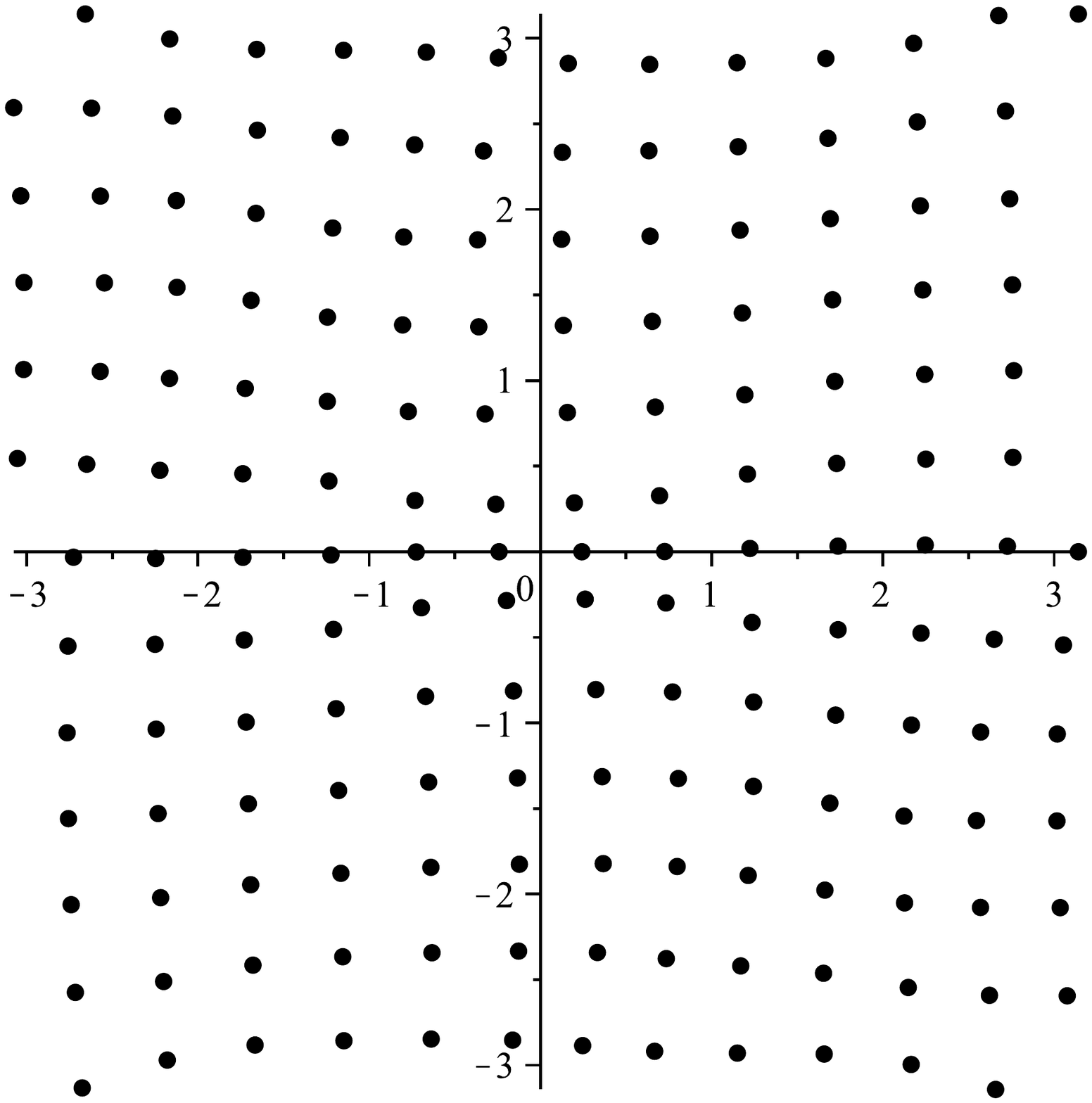}&\hspace*{6mm}& \includegraphics[scale=0.30]{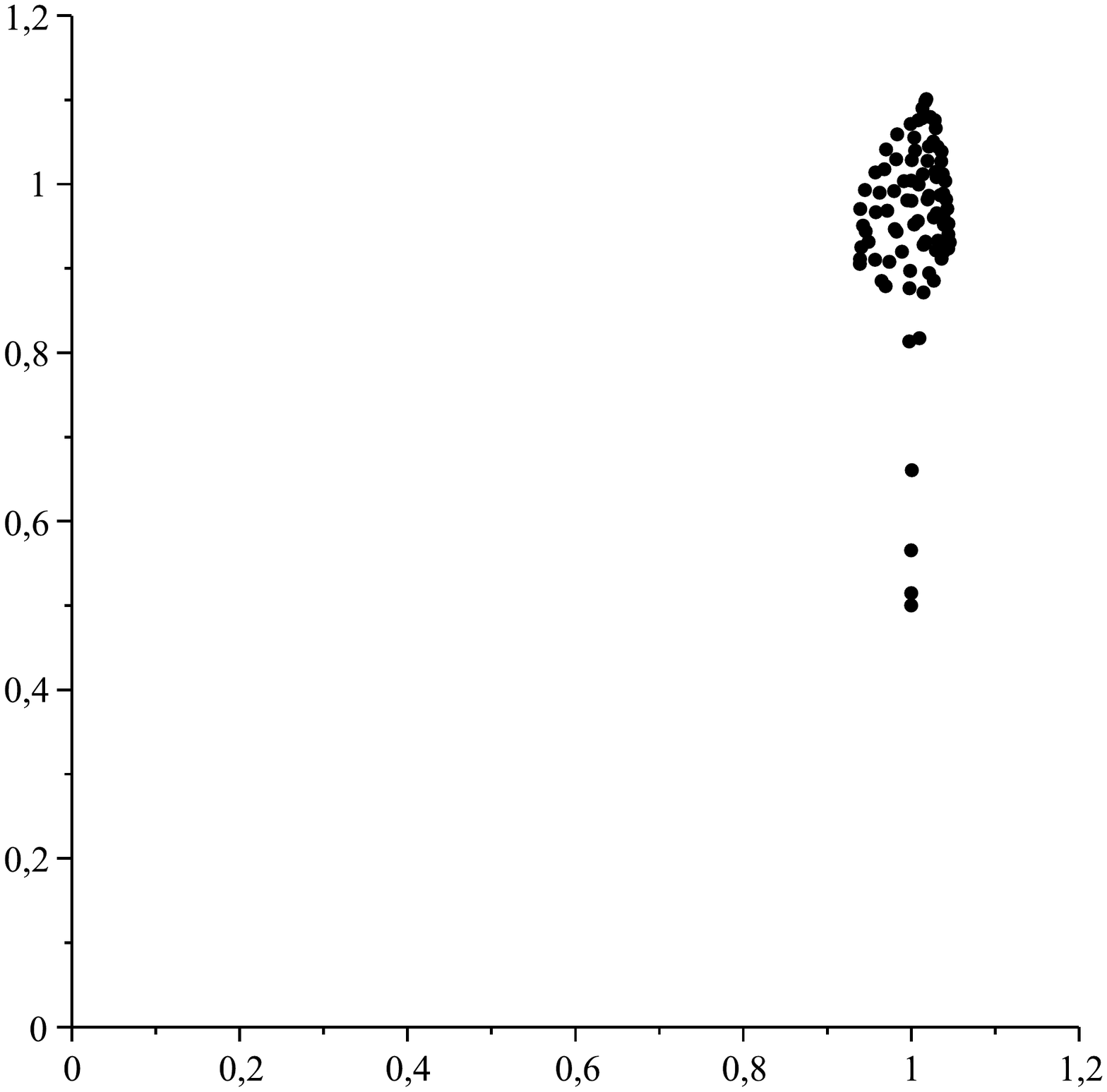}    
  \end{tabular}
 \vspace{-3mm}\caption{Angle and radius distribution of the zeros of
   the system \eqref{33}}\label{fig:6}
\end{figure}
This example and others of the same kind suggest that Theorem
\ref{mtunivariate} has an extension to higher dimensions.  

The study of the distribution of the solutions of a system of
multivariate polynomial equations has been addressed from different
perspectives. For instance, Khovanskii's theorem on complex fewnomials
~\cite[\S 3.13, Theorem~2]{Kho91} gives an estimate for the
distribution of the arguments of these solutions in terms of the
number of monomials and the Newton polytopes of the input
system. There are also several interesting results by Shiffman,
Zelditch and Bloom on the asymptotic distribution of the solutions of
a random system of polynomial equations, see for instance \cite{SZ04,
  BS07} and the references therein.

Our purpose in this text is to obtain an extension of
Theorem \ref{mtunivariate} to systems of Laurent polynomials
with a given support. For $i=1,\dots,n$, let $\cA_i$ be a non-empty
finite subset of $\Z^{n}$ and $Q_{i}=\conv(\cA_{i})\subset\R^{n}$ its convex
hull. Set $D=\MV_{\R^{n}}(Q_{1},\dots,Q_{n})$ for the mixed volume of these
lattice polytopes, and assume that $D\ge1$.  For each $i$,
let $f_{i}$ be a Laurent polynomial with support contained in
$\cA_{i}$, that is,
\begin{displaymath}
  f_i=\sum_{\bfa\in \cA_{i}}\alpha_{i,\bfa} \, \bfx^{\bfa}\in
\C[x_1^{\pm1},\dots, x_n^{\pm1}]
\end{displaymath}
with $\alpha_{i,\bfa}\in \C$ and $\bfx^{\bfa}=x_{1}^{a_{1}}\dots
x_{n}^{a_{n}}$ for each $\bfa=(a_{1},\dots,a_{n})\in \cA_{i}$.  We
write~$\bff=(f_{1},\dots, f_{n})$ for short.  We assume that, for all vectors
$\bfv\in \Z^{n}$, the directional resultant
$\Res_{\cA_{1}^{\bfv},\dots,\cA_{n}^{\bfv}}(f_{1}^{\bfv},\dots,f_{n}^{\bfv})$ 
(Definition \ref{def:1}) is nonzero. This condition holds for a
generic choice of~$\bff$ in the space of coefficients and, by
Bernstein's theorem \cite[Theorem~B]{Ber75}, it implies that all the
solutions of $f_{1}=\dots=f_{n}=0$ are isolated and that their number,
counted with multiplicities, is equal to $D$.

For a vector $\bfw\in S^{n-1}$ in the unit sphere of $\R^n$, we denote
by $\bfw^{\bot}\subset \R^{n}$ its orthogonal subspace and by
$\pi_{\bfw}\colon \R^{n}\to \bfw^{\bot}$ the corresponding orthogonal
projection. We denote by $\MV_{\bfw^{\bot}}$ the mixed volume of convex
bodies of $\bfw^{\bot}$ induced by the Euclidean measure on
$\bfw^{\bot}$ and, for $i=1,\dots,n$, we set
\begin{displaymath}
D_{\bfw,i}=\MV_{\bfw^\bot}\big(\pi_\bfw(Q_1),\dots,
\pi_\bfw(Q_{i-1}),\pi_\bfw(Q_{i+1}),\dots,\pi_\bfw(Q_n)).
\end{displaymath}

We then define the {\em Erd\"os-Tur\'an size} of  $\bff$ as
\begin{equation} \label{eq:1}
  \eta(\bff)= \frac{1}{D} \sup_{\bfw\in S^{n-1}}
\log\Bigg(\frac{\prod_{i=1}^{n}\|f_i\|_{\sup}^{D_{\bfw,i}}} 
 { \prod_{\bfv}|\Res_{\cA_{1}^{\bfv},\dots,\cA_{n}^{\bfv}}(f_{1}^{\bfv},\dots,f_{n}^{\bfv})|^{\frac{|\langle
          \bfv,\bfw \rangle|}{2}}}\Bigg),
\end{equation}
where the second product is over all primitive vectors $\bfv\in\Z^n$
that is, vectors whose coordinates do not have a non-trivial common
factor, and $\langle \cdot,\cdot\rangle$ is the standard inner product
of $\R^n$.  This product is finite because
$\Res_{\cA_{1}^{\bfv},\dots,\cA_{n}^{\bfv}}(f_{1}^{\bfv},\dots,f_{n}^{\bfv})\neq
1$ only if~$\bfv$ is an inner normal to a facet of the Minkowski sum
$Q_{1}+\cdots +Q_{n}$.  

The Erd\"os-Tur\'an size is a generalization
to the multivariate case of the quantity
$\frac{1}{d}\log\Big(\frac{\|f\|_{\sup}}{\sqrt{|a_0a_d|}}\Big)$ that
appears in Theorem \ref{mtunivariate} since, for $n=1$, it is easily
checked that $\eta(\bff)$ is exactly the preceding quantity
(Proposition \ref{prop:2}).

Let $Z(\bff)$ denote the 0-dimensional effective cycle
of~$(\C^\times)^n$ defined by the roots of~$\bff$ and  their multiplicities.  The
angle and radius discrepancies of cycles of $(\C^{\times})^{n}$ are
the obvious generalization of those for the univariate case (see
Definition~\ref{def:3}).

Our main result is the following:

\begin{theorem} \label{thm:1} For $n\ge2$, let $\cA_{1},\dots,\cA_{n}$ be non-empty
  finite subsets of $ \Z^{n} $, set $Q_{i}=\conv(\cA_{i})$ and assume
  that $\MV_{\R^{n}}(Q_{1},\dots,Q_{n})\ge1$.  Let $f_{1},\dots, f_{n}\in
  \C[x_{1}^{\pm1},\dots, x_{n}^{\pm1}]$ with $\supp(f_{i})\subset
  \cA_{i}$ and such that
  $\Res_{\cA_{1}^{\bfv},\dots,\cA_{n}^{\bfv}}(f_{1}^{\bfv},\dots,f_{n}^{\bfv})\ne
  0$ for all $\bfv\in \Z^{n}\setminus \{\bfzero\}$. Then
  \begin{equation*}
\Deltaa(Z(\bff))\leq 66\,n\, 2^{n} ( 18+
\log^{+}({\eta(\bff)^{-1}}))^{\frac23(n-1)} \eta(\bff)^{\frac13}
  \end{equation*}
with $\log^{+}(x)=\log(\max(1,x))$ for $x>0$. 
Also, for $0<\varepsilon <1,$
\begin{equation*}
\Deltar(Z(\bff),\varepsilon) \le  \frac{2n}{\varepsilon} \eta(\bff).  
\end{equation*}
\end{theorem}

Theorem \ref{mtunivariate} shows that these bounds for the angle and the
radius discrepancy also hold in the case $n=1$.  By analogy with
the one-dimensional case, it is natural to ask if, in the setting of
our result, a stronger inequality of the form
\begin{displaymath}
  \Deltaa(Z(\bff))\leq
c(n)\,\eta(\bff)^{\frac12}
\end{displaymath}
holds, with $c(n)>0$ not depending on $\bff$. It would be interesting to settle this question.

\medskip
Theorem \ref{thm:1} has several  consequences. For instance, we can
derive from it a bound for the number of positive real solutions of a
 system of polynomial equations, in terms of its Erd\"os-Tur\'an
size.  For a cycle $Z=\sum_{\bfxi}m_{\bfxi}[\bfxi] $ of $(\C^{\times})^{n}$,
set
\begin{displaymath}
  Z_{+}= \sum_{\bfxi \in (\R_{>0})^{n}}m_{\bfxi}[\bfxi].
\end{displaymath}
 
The following statement follows immediately from Theorem \ref{thm:1} and
the definition of the angle discrepancy.

\begin{corollary}
  \label{cor:3}
Let  notation be as  in Theorem \ref{thm:1}. Then
\begin{displaymath}
\deg(Z(\bff)_{+})\leq 66\,n\, 2^{n} ( 18+
\log^{+}({\eta(\bff)^{-1}}))^{\frac23(n-1)} \eta(\bff)^{\frac13} 
\, \deg(Z(\bff)).
\end{displaymath}
\end{corollary}


We can also apply our result to study the asymptotic distribution of
the roots of a sequence of systems of polynomials over $\Z$ with
growing supports and whose coefficients are not too big. To be more
precise, let $Q_i$, $i=1,\dots, n$, be lattice polytopes in $\R^{n}$
such that $\MV_{\R^{n}}(Q_1,\ldots, Q_n)\ge1$. For each integer $\kappa\ge 1$
and $i=1,\dots, n$, consider the finite subset of $\Z^{n}$ given by
\begin{equation}
  \label{eq:3}
\cA_{\kappa, i}=\kappa Q_i\cap\Z^n .
\end{equation}
For a  Laurent
polynomial $f\in \C[x_{1}^{\pm1},\dots,x_{n}^{\pm1}]$, we denote by
$\supp(f)$ its  support, defined as the subset of $\Z^{n}$ of its
exponent vectors. We also set 
\begin{displaymath}
 \|f\|_{\sup}=\sup_{|w_{1}|=1,\dots, |w_{n}|=1}|f(w_{1},\dots,w_{n})|.    
\end{displaymath}

For a nonzero cycle $Z=\sum_{\bfxi}m_{\bfxi}[\bfxi]$ of
$(\C^{\times})^{n}$, we consider the discrete probability measure on
$(\C^{\times})^{n}$ defined by
$$
\delta_{Z}= \frac{1}{\deg(Z)}\sum_{\bfxi}m_{\bfxi} \delta_\bfxi ,
$$
where $\delta_\bfxi$ is the Dirac measure supported on the point
$\bfxi$.  Let $\nu_{\haar}$ be the measure on~$(\C^{\times})^{n}$ supported on
$(S^{1})^{n}$ and whose restriction to this polycircle coincides
with its Haar measure of total mass 1.  

Recall that a sequence of measures $(\nu_{\kappa})_{\kappa\in\N}$
on~$(\C^{\times})^{n}$ converges weakly to $\nu_{\haar}$ if, for every
continuous function with compact support~$h\colon (\C^{\times})^{n}\to
\R$, it holds
\begin{displaymath}
  \lim_{\kappa\to\infty}\int_{(\C^{\times})^{n}}h \, \d\nu_{\kappa}=\int_{(\C^{\times})^{n}}h \, \d\nu_{\haar}.
\end{displaymath}
If this is the case, we write $\lim_{\kappa\to \infty}\nu_{\kappa}=\nu_{\haar}$. 

\begin{theorem} \label{cor:2} For $\kappa\ge 1$, let $\bff_{\kappa}=(f_{\kappa, 1},\dots,
  f_{\kappa, n}) $ be a family of Laurent polynomials in $\Z[x_1^{\pm1},\dots,
  x_n^{\pm1}] $ such that $ \supp(f_{\kappa, i})\subset \kappa Q_{i}$, $
  \log\|f_{\kappa, i}\|_{\sup}= o(\kappa)$, and 
  $\Res_{\cA_{\kappa, 1}^{\bfv},\dots,\cA_{\kappa, n}^{\bfv}}(f_{\kappa, 1}^{\bfv},\ldots,f_{\kappa, n}^{\bfv})\ne0$
  for  all $\bfv\in \Z^{n}\setminus \{\bfzero\}$. Then
\begin{displaymath}
  \lim_{\kappa\to\infty}\delta_{Z(\bff_{\kappa})} = \nu_{\haar}.
\end{displaymath}
\end{theorem}

This result admits a quantitative version giving information on the
rate of convergence, which we state in Proposition \ref{prop:7}.
Theorem~\ref{cor:2} is related to Bilu's equidistribution theorem for
the Galois orbit of algebraic points in $(\C^{\times})^{n}$ of small
height \cite{Bil97} which, at least for $n=1$, also admits
quantitative versions \cite{Pet05,FR06}. 


\medskip We can also apply Theorem \ref{thm:1} to study the
distribution of the roots of a random system of Laurent polynomials
over $\C$. We will show that, under some mild conditions and
without assuming any independence or equidistribution condition on the
coefficients of the system, these roots tend to cluster uniformly near
$(S^{1})^{n}$.

To state this result, let us keep notation as above and set
$\bfcA_{\kappa}=(\cA_{\kappa,1},\dots, \cA_{\kappa,n})$ with $\cA_{\kappa,i}$ as
in~\eqref{eq:3}.  Each point of the projective space
$\P(\C^{\bfcA_{\kappa}})$ can be identified with a system
$\bff_\kappa=(f_{\kappa,1},\dots,f_{\kappa,n})$ of Laurent polynomials such that
$\supp(f_{\kappa,i})\subset \kappa Q_{i}$, $i=1,\dots, n$, modulo a
multiplicative scalar. The associated cycle $Z(\bff_\kappa)$ is
well-defined, since it does not depend on this multiplicative scalar.

Let $\mu_{\kappa}$ be the normalized Fubini-Study measure on
$\P(\C^{\bfcA_{\kappa}})$ of total mass 1, and $\lambda_{\kappa}$ a probability
density function on $ \P(\C^{\bfcA_{\kappa}})$ (see \S \ref{sec:equid-prob}
for details).  Let $\bff_\kappa$ be a random system of Laurent polynomials
with $\supp(f_{\kappa,i})\subset \kappa Q_{i}$, $i=1,\dots, n$, distributed
according to the probability law given by $\lambda_{\kappa}$ with respect to
$\mu_{\kappa}$.  The \emph{expected zero density measure} of $\bff_\kappa$ is the
measure on $(\C^{\times})^{n}$ defined, for a Borel subset $U$, as
\begin{displaymath}
  \E(Z(\bff_\kappa);
  \lambda_\kappa)(U)=\int_{\P(\C^{\bfcA_\kappa})}{\deg(Z(\bff_\kappa)|_{U})}\, \lambda_{\kappa}(\bff_\kappa)\,\d\mu_\kappa,
\end{displaymath}
where  $Z(\bff_\kappa)|_{U}$ denotes the  cycle $\sum_{\bfxi\in V(\bff_\kappa)_{0}\cap U}m_\bfxi[\bfxi].$

\begin{theorem} \label{thm:3} For $\kappa\ge1$, let $\lambda_{\kappa}$ be
  a probability density function on $\P(\C^{\bfcA_{\kappa}})$ with
  respect to the measure $\mu_{\kappa}$, and
  $\bff_\kappa=(f_{\kappa,1},\dots,f_{\kappa,n})$ a random system of
  Laurent polynomials with $\supp(f_{\kappa,i})\subset \kappa Q_{i}$,
  $i=1,\dots, n$, distributed according to the probability law given
  by $\lambda_{\kappa}$. Assume that the sequence
  $(\lambda_{\kappa})_{\kappa\ge1}$ is uniformly bounded. Then
\begin{displaymath}
  \lim_{\kappa\to\infty}\frac{\E(Z(\bff_\kappa);{\lambda_{\kappa}})}{\kappa^{n}\MV_{\R^{n}}(Q_{1},\dots,Q_{n})} =\nu_{\haar}.
\end{displaymath}
\end{theorem}

As an application, consider a random system of Laurent polynomials
$\bff_\kappa$ with $\supp(f_{\kappa, i})\subset \kappa Q_{i}$ whose coefficients are
independent complex Gaussian random variables with mean 0 and variance
1.  The random cycle $Z(\bff_\kappa )$ might be described by the uniform
distribution on $\P(\C^{\bfcA_{\kappa}})$ (see Example \ref{exm:1} for
details).  Then, Theorem~\ref{thm:3} implies that the roots of
$\bff_\kappa $ converge weakly to the equidistribution on $(S^{1})^{n}$, and
we recover in this way a result of Bloom and Shiffman
\cite[Example~3.5]{BS07}.

\medskip Our strategy for proving Theorem \ref{thm:1} consists of
reducing to the univariate case.  In \S \ref{sec:angle-radi-distr}, we
consider the problem of studying the angle and radius discrepancies of
an arbitrary 0-dimensional effective cycle $Z$ in $(\C^{\times})^{n}$
in terms of the angle and radius discrepancies of its direct images
under all monomial projections of $(\C^{\times})^{n}$
onto~$\C^{\times}$. By applying a tomography process based on Fourier
analysis, we show that the distribution of $Z$ can be controlled in
terms of the distribution of its projections (Theorem~\ref{thm:2}).
In~\S \ref{sec:bounds-discr-terms}, we consider cycles defined by a
system of Laurent polynomials with given support and we compute their
direct image under monomial projections, in terms of sparse
resultants. Theorem \ref{thm:1} then follows by applying
Erd\"os-Tur\'an's theorem combined with Theorem~\ref{thm:2}, and the
basic properties of the sparse resultant.

In \S~\ref{sec:equid-prob}, we study the asymptotic distribution of
the roots  of a sequence of systems of Laurent polynomials over $\Z$ and of
random systems of Laurent polynomials over~$\C$. 
In both situations, the key step consists of bounding  from below the size of
the relevant directional resultants. In the case of systems over $\Z$, this
is trivial since these directional resultants are nonzero integer
numbers. In the case of random systems over $\C$, the result follows
from  an estimate of the volume of a tube around an algebraic
variety due to Beltr\'an and Pardo~\cite{BP:edcnsm}.


\vspace{3mm}
\noindent{\bf Acknowledgments.} 
We thank Carlos Beltr\'an and Michael Shub for useful discussions and
pointers to the literature. Part of this work was done while the
authors met at Universitat de Barcelona and Universit\'e de
Nice--Sophia Antipolis. We thank these institutions for their
hospitality.

\section{Angle and radius distribution in the multivariate case}\label{sec:angle-radi-distr}

In this section, we show that the angle and radius discrepancies of an
effective 0-dimensional cycle in the algebraic torus $(\C^{\times})^{n}$ can be
bounded in terms of the angle discrepancy of its image under monomial
maps from $(\C^{\times})^{n}$ to $\C^{\times}$.

Let $Z$ be a nonzero effective 0-dimensional
cycle of $(\C^{\times})^{n}$, which we write as a finite sum
\begin{displaymath}
  Z=\sum_{\bfxi}m_{\bfxi}[\bfxi] 
\end{displaymath}
with $m_{\bfxi}\in\N$ and $\bfxi\in (\C^\times)^n$.
The \emph{support} of $Z$ is the finite subset of $(\C^{\times})^{n}$
defined as $|Z|=\{\bfxi\mid m_{\bfxi}\ge 1\}$, and the \emph{degree} of $Z$ is the positive number
$\deg(Z)=\sum_\bfxi m_\bfxi.$

\begin{definition} \label{def:3} 
Let $  Z=\sum_{\bfxi}m_{\bfxi}[\bfxi] $ be a nonzero effective 0-dimensional
cycle of $(\C^{\times})^{n}$.
For each  $\bfalpha=(\alpha_{1},\dots, \alpha_{n})$ and
$\bfbeta=(\beta_{1},\dots, \beta_{n})$ with $-\pi\le
\alpha_{j}<\beta_{j}\le \pi$, $j=1,\dots, n$, consider the cycle 
\begin{displaymath}
  Z_{\bfalpha,\bfbeta}=\sum_{-\alpha_{j}<\arg(\xi_{j})\le \beta_{j}} m_{\bfxi}[\bfxi].
\end{displaymath}
The
{\em angle discrepancy} of $Z$ is defined as
\begin{equation*}
  \Deltaa(Z)= \sup_{\bfalpha, \bfbeta}\bigg|  \frac{\deg(Z_{\bfalpha, \bfbeta})}{\deg(Z)}-
  \prod_{j=1}^{n}\frac{\beta_{j}-\alpha_{j}}{2\pi}\bigg|.
\end{equation*}
Let $0<\varepsilon <1$ and consider also the cycle 
\begin{displaymath} 
  Z_{\varepsilon}= \sum_{1-\varepsilon<|\xi_j|<(1-\varepsilon)^{-1}} m_{\bfxi}[\bfxi],
 \end{displaymath}
where $\xi_{j}$ is the $j$-th coordinate of $\bfxi$.
The {\em radius discrepancy} of $Z$ with
respect to $\varepsilon$ is defined as
$$
\Deltar(Z, \varepsilon)=1- \frac{\deg(Z_{\varepsilon})}{\deg(Z)}.
$$ 
\end{definition}

We have $0< \Deltaa(Z)\leq1$.  Observe also that $0\le
\Deltar(Z,\varepsilon)\leq1$ and $\Deltar(Z, \varepsilon)=0$ for
all $\varepsilon$ if and only if $|Z|\subset (S^{1})^{n}$.

For a lattice point $\bfa=(a_{1},\dots,a_{n})\in \Z^{n}$ we denote by
$\chi^{\bfa}\colon (\C^{\times})^{n}\to \C^{\times}$ the associated
character, defined
as $\chi^{\bfa}(\bfxi)=\xi_{1}^{a_{1}}\dots\xi_{n}^{a_{n}}$ for $\bfxi\in(\C^{\times})^{n}$.
The direct image of $Z$ under $\chi^{\bfa}$ is the cycle of
$\C^{\times}$ given by 
\begin{displaymath}
  \chi^{\bfa}_{*}(Z)= \sum_{\bfxi} m_{\bfxi}\chi^{\bfa}(\bfxi). 
\end{displaymath}

We also set
\begin{equation}\label{theta}
  \theta(Z)=\sup_{\bfa\in \Z^{n}\setminus \{\bfzero\}}
  \frac{\Deltaa(\chi^{\bfa}_{*}(Z))}{\|\bfa\|_{2}^{\frac12}}, \quad 
  \rho(Z,\varepsilon)= \sum_{j=1}^{n}\Deltar(\chi^{\bfe_{j}}_{*}(Z), \varepsilon),
\end{equation}
where $\bfe_{j}$ denotes the $j$-th vector in the standard basis of
$\Z^{n},$ and $\|\bfa\|_2$ is the Euclidean norm of the vector $\bfa\in\Z^n$.
We have $0<\theta(Z)\leq1$ and $ 0\le \rho(Z,\varepsilon)\le 1$. 

\begin{theorem} \label{thm:2}
Let $Z$ be a nonzero effective 0-dimensional
cycle of $(\C^{\times})^{n}$. Then
$$
\Deltaa(Z)\le 22n\bigg(\frac83\bigg)^{n}
(9-\log(\theta(Z)))^{\frac23(n-1)} \, \theta(Z)^{\frac23}
$$
and, for $0<\varepsilon <1$, 
\begin{displaymath}
  \Deltar(Z ,\varepsilon)
\le \rho(Z,\varepsilon).
\end{displaymath}
\end{theorem}

The rest of this section is devoted to the proof of this result.
Given two vectors
$\bfu=(u_{1},\dots,u_{n}),\bfv=(v_{1},\dots,v_{n})\in \R^{n}$ we write
$\langle\bfu,\bfv\rangle=\sum_{j=1}^{n}u_{j}v_{j}$ for their standard
inner product, and for $\bfxi\in(\C^\times)^n$ we set $\arg(\bfxi)=
(\arg(\xi_{1}),\dots, \arg(\xi_{n}))\in (-\pi,\pi]^n$.

\begin{lemma}
  \label{lemm:1}
Let $Z$ be a nonzero effective 0-dimensional
cycle of $(\C^{\times})^{n}$ and $\bfa\in \Z^{n}\setminus \{\bfzero\}$. Then 
\begin{displaymath}
  \Bigg|\frac{1}{\deg(Z)} \sum_{\bfxi} m_{\bfxi}\e^{\i \langle \bfa,
    \arg(\bfxi)\rangle}\Bigg| \le 2\pi\Deltaa(\chi^{\bfa}_{*}(Z)).
\end{displaymath}
\end{lemma}

Note that for $\bfa={\bfzero}$ we get  $\frac{1}{\deg(Z)}\sum_{\bfxi} m_{\bfxi}\e^{\i \langle {\bfa}, \arg(\bfxi)\rangle}=1$.

\begin{proof}
  Set $D_{0}=\#|Z|$ and $D=\deg(Z)$ for short.  Let $-\pi\leq \nu_{k}<\pi$,
  $k=1,\dots, D_{0}$, denote the inner products $\langle \bfa,
  \arg(\bfxi)\rangle$ modulo $2\pi$ for the different points $\bfxi$ in the support
  of $Z$, and let $m_{k}$ denote their corresponding multiplicity. We
  suppose that these numbers are arranged in increasing order, that
  is, $\nu_{1}\le \dots\le \nu_{D_{0}}$.

For $-\pi<\nu\le \pi$ set
$$
N(\nu)=\sum_{k\mid \nu_k\leq\nu}m_{k}.
$$
We have 
\begin{multline} \label{eq:2}
 \int_{-\pi}^{\pi}N(\nu)\e^{\i\nu}d\nu =\sum_{k=1}^{D_{0}}\int_{\nu_{k}}^{\nu_{k+1}}
 \bigg(\sum_{l\le k}m_{l}\bigg)\e^{\i\nu} d\nu \\=
 -\i\sum_{k}\bigg(\sum_{l\le k}m_{l}\bigg)\,\e^{\i\nu}\bigg|_{\nu_{k}}^{\nu_{k+1}}=\i\Big(D+\sum_{k}m_{k}\e^{\i\nu_k}\Big),
\end{multline}
where we have set $\nu_{D_{0}+1}=\pi$.
On the other hand, an easy calculation shows that
\begin{equation}\label{dos}
  \int_{-\pi}^{\pi}\frac{\nu+\pi}{2\pi}
  \e^{\i\nu}d\nu=\i.
\end{equation}
Combining (\ref{eq:2}) and (\ref{dos}), we deduce that
$$
\frac{1}{D} \sum_{\bfxi} m_{\bfxi}\e^{\i \langle \bfa,
    \arg(\bfxi)\rangle}=\frac{1}{D}\sum_{k=1}^{D_{0}}m_{k}\e^{\i\nu_k}=\i\int_{-\pi}^{\pi}\left(\frac{\nu+\pi}{2\pi}-\frac{N(\nu)}{D}\right)\e^{\i\nu}d\nu.
$$
Hence
\begin{multline*}
\left|\frac{1}{D} \sum_{\bfxi} m_{\bfxi}\e^{\i \langle \bfa,
    \arg(\bfxi)\rangle}\right|\leq\int_{-\pi}^{\pi}\left|\frac{\nu+\pi}{2\pi}-\frac{N(\nu)}{D}\right|d\nu
\\=\int_{-\pi}^{\pi}\left|\frac{\nu+\pi}{2\pi}-\frac{\deg(\chi^{\bfa}_{*}(Z)_{-\pi,\nu})}{\deg(\chi^{\bfa}_{*}(Z))}\right|d\nu\le 2\pi\Deltaa(\chi^{\bfa}_{*}(Z)),
\end{multline*}
which concludes the proof.
\end{proof}

Let $\alpha,\beta,\tau\in\R$ such that $\alpha\leq\beta$ and
$\tau>0$. We consider the function
$h_{\alpha,\beta,\tau}\colon \R\to \R$ defined, for $x\in \R$, by 
$$
  h_{\alpha,\beta,\tau}(x)=
  \begin{cases}
0& \mbox{if }\, x\leq\alpha-\tau,\\
g\big(\frac{x-\alpha+\tau}{\tau}\big)&\mbox{if}\,\alpha-\tau<x\leq\alpha,\\
1&\mbox{if }\, \alpha<x\leq\beta,\\    
g\big(\frac{\beta+\tau-x}{\tau}\big) &\mbox{if }\,\beta<x\leq\beta+\tau,\\
0&\mbox{if }\, \beta+\tau<x,                                 
  \end{cases}
$$
with $g(x)=-2x^3+3x^2$.
Lemma \ref{univariate} below shows that  $h_{\alpha,\beta,\tau}$ is an approximation of the
characteristic function of the interval $[\alpha, \beta]$. 

For $m\in\N$, we denote by ${\mathcal C}^m(\R)$ the space of functions $f\colon \R\to\R$ having $m$~continuous derivatives.

\begin{lemma}\label{univariate} 
Let $\alpha,\beta,\tau\in\R$ such that $\alpha\leq\beta$ and
$\tau>0$.  Then
\begin{enumerate}
 \item \label{item:1} $h_{\alpha,\beta,\tau}\in{\mathcal C}^1(\R)$;
\item \label{item:2} $h_{\alpha,\beta,\tau}(x)=1$ for $x\in
  [\alpha, \beta]$, $h_{\alpha,\beta,\tau}(x)=0$ for $x\in
  (-\infty, \alpha - \tau] \cup [\beta+\tau, \infty)$,
  and  
$0\leq h_{\alpha,\beta,\tau}(x)\leq1$ for all $x\in\R$;
\item \label{item:3} $\int_{-\infty}^{\infty}
  h_{\alpha,\beta,\tau}dx=\beta-\alpha+\tau$ and,
  moreover,  
$\int_{\alpha-\tau}^\alpha
h_{\alpha,\beta,\tau}dx=\int_\beta^{\beta+\tau}h_{\alpha,\beta,\tau}dx=\frac{\tau}{2}$;
\item \label{item:4} $\int_{-\infty}^{\infty}|h'_{\alpha,\beta,\tau}|dx=2$;
\item \label{item:5} $\int_{-\infty}^{\infty}|h''_{\alpha,\beta,\tau}|dx=\frac{6}{\tau}.$
\end{enumerate}
\end{lemma}

\begin{proof}
By a direct calculation, we verify that the function $g$ satisfies the following properties:
\begin{itemize}
\item $g(x)\geq0$ for all $x\in[0,1]$;
 \item $g(0)=g'(0)=0,\ g(1)=1,\,g'(1)=0$;
\item $\int_0^1g\,dx=\frac12$;
\item $\int_0^1|g'|\,dx=1$;
\item $\int_0^1|g''|\,dx=3.$
\end{itemize}
The claim follows easily from these properties and the definition of
$h_{\alpha,\beta,\tau}$.
\end{proof}

Suppose furthermore that $\beta-\alpha+2\tau<2\pi.$ The support of
$h_{\alpha,\beta,\tau}$ is then contained in an interval of length
bounded by $2\pi$, and so this function can be regarded as a function
on $\R/2\pi\Z$. For $a\in \Z$ set
$c_{a}=\frac{1}{2\pi}\int_{-\pi}^{\pi}h_{\alpha,\beta,\tau}(x)\e^{-\i
  ax}\d x$, so that its Fourier series is given by
$$\sum_{a\in\Z}c_{a}\e^{\i ax}.
$$

\begin{lemma}\label{acotacionFourier} Let
  $\alpha\le \beta$ and $\tau>0$ such that
  $\beta-\alpha+2\tau<2\pi$. Then $\sum_{a\in\Z}c_{a}\e^{\i ax}$
  converges absolutely and uniformly on $\R/2\pi\Z$ to
  $h_{\alpha,\beta,\tau}$.  Moreover,
  $c_0=\frac{\beta-\alpha+\tau}{2\pi}, $ and, for $a\neq0$,
\begin{equation}\label{eq:7}
|c_a|\le \min \Big\{ \frac{1}{\pi\,a}, \frac{3}{\pi\,\tau\,a^2}\Big\}.  
\end{equation}

\end{lemma}
\begin{proof}
  Lemma \ref{univariate}\eqref{item:1} implies that the series
  $\sum_{a\in\Z}c_{a}\e^{\i ax}$ converges absolutely and uniformly on
  $\R/2\pi\Z$ to the function $h_{\alpha,\beta,\tau}$. The computation
  of $c_0$ follows from Lemma
  \ref{univariate}\eqref{item:3}. Integrating by parts, we deduce for
  $a\in\Z\setminus \{0\}$ that
$$
c_a=\frac{1}{-2\pi
  \i a}\int_{-\pi}^{\pi}h'_{\alpha,\beta,\tau}(x)\e^{-\i
  ax}dx=\frac{1}{2\pi(-\i a)^2}\int_{-\pi}^{\pi}h''_{\alpha,\beta,\tau}(x)\e^{-\i
  ax}dx.
$$
Hence, $|c_a|\le \frac{1}{2\pi
  a}\int_{-\pi}^{\pi}|h'_{\alpha,\beta,\tau}|dx$ and also $|c_a|\le
\frac{1}{2\pi a^2}\int_{-\pi}^{\pi}|h''_{\alpha,\beta,\tau}|dx$. Then \eqref{eq:7}
follows by bounding these integrals with Lemma
\ref{univariate}(\ref{item:4}-\ref{item:5}).
\end{proof}

Next, we apply Fourier analysis to control the angle discrepancy of $Z$
in terms of the angle discrepancy of its direct image under monomial
projections.

\begin{lemma}\label{est33}
  Let $n\ge 2,\, q \in \Z_{\ge 1}$, $\bfalpha=(\alpha_{1},\dots, \alpha_{n})$ and
  $\bfbeta=(\beta_{1},\dots, \beta_{n})$ with $\alpha_{j},\beta_{j}\in
  \R$ such that $-\pi\leq\alpha_j<\beta_j<\pi$ and
  $\beta_{j}-\alpha_{j}+\frac{2}{q }< 2\pi$.  Then
\begin{displaymath}
  \left|\frac{\deg(Z_{\bfalpha,\bfbeta})}{\deg(Z)}-\prod_{j=1}^n\frac{\beta_j-\alpha_j}{2\pi}\right|
  \le 2\, n\, \theta(Z) +\frac{3\, n}{2\, \pi\, q } 
+n\frac{2^{n+3}\sqrt{3}}{\pi^{n-1}}
  q ^{\frac12}(9+
  \log({q }))^{n-1}  \, \theta(Z).
\end{displaymath}
\end{lemma}

\begin{proof}
  Set $\tau= \frac{1}{q }$ and 
  $h_{\bfalpha,\bfbeta,\tau}(\bfnu)=\prod_{j=1}^nh_{\alpha_j,\beta_j,\tau}(\nu_j)$
  for $\bfnu=(\nu_{1},\dots, \nu_{n})\in \R^{n}$. Set also $D=\deg(Z)$ and
\begin{align*}
  \Sigma_1&=\left|\frac{\deg(Z_{\bfalpha,\bfbeta})}{D}-\frac{1}{D}\sum_{\bfxi}m_{\bfxi}h_{\bfalpha,\bfbeta,\tau}(\arg(\bfxi))\right|,\\   
\Sigma_2&=\left|\frac{1}{D}\sum_{\bfxi}m_{\bfxi}h_{\bfalpha,\bfbeta,\tau}(\arg(\bfxi))-\prod_{j=1}^n\frac{\beta_j-\alpha_j+\tau }{2\pi}\right|,\\ 
\Sigma_3&=\left|\prod_{j=1}^n\frac{\beta_j-\alpha_j+\tau }{2\pi}
  - \prod_{j=1}^n\frac{\beta_j-\alpha_j}{2\pi}\right|. 
\end{align*}
We will bound each of these quantities.  For $\Sigma_{1}$,
we consider the subset of $\R^{n}$ given by
 $I_{\bfalpha,\bfbeta,\tau}=\prod_{j=1}^n[\alpha_{j}-\tau,\beta_{j}+\tau]\setminus\prod_{j=1}^n
[\alpha_{j},\beta_{j}]$. Then
$$
\Sigma_1=\Bigg|\frac{1}{D}\sum_{\arg(\bfxi)\in I_{\bfalpha,\bfbeta,\tau}} m_{\bfxi}
h_{\bfalpha,\bfbeta,\tau}(\arg(\bfxi))\Bigg|.
$$
For each $\bfxi$ such that $\arg(\bfxi) \in I_{\bfalpha,\bfbeta,\tau}$, there is $1\leq j\leq n$ such that either  $\alpha_j-\tau <
\arg(\xi_{j})\le \alpha_{j}$ or
$\beta_{j}<\arg(\xi_{j})<\beta_{j}+\tau.$
Since $0\leq h_{\bfalpha,\bfbeta,\tau}(\bfnu)\leq 1$ for all
$\bfnu,$ we have that~$\Sigma_{1}$ is bounded from above by
$$
\frac{1}{D}\sum_{j=1}^n\big(\deg(\chi^{\bfe_{j}}_{*}(Z)_{\alpha_{j}-\tau ,\alpha_{j}})+
  \deg(\chi^{\bfe_{j}}_{*}(Z)_{\beta_{j}, \beta_{j}+\tau })\big).
$$
By using the definition of $\Deltaa$ and Lemma \ref{lemm:1}, we get
\begin{displaymath}
  \frac{1}{D}\deg(\chi^{\bfe_{j}}_{*}(Z)_{\alpha_{j}-\tau ,\alpha_{j}})\le
  \Deltaa(\chi^{\bfe_{j}}_{*}(Z))+ \frac{\tau }{2\pi}\le
  \theta(Z)+ \frac{\tau }{2\pi}=\theta(Z)+ \frac{1}{2\pi q },
\end{displaymath}
and a similar bound holds for $\deg(\chi^{\bfe_{j}}_{*}(Z)_{\beta_{j},
  \beta_{j}+\tau })$. Hence, 
\begin{equation}\label{s1}
\Sigma_1\leq 2n\theta(Z) +\frac{n}{\pi q }.
\end{equation}

Now we turn to $\Sigma_{2}$. Due to the conditions imposed on $\tau$, we can
regard~$h_{\bfalpha,\bfbeta,\tau}$ as a function on~$\R^{n}/2\pi
\Z^{n}\simeq (-\pi,\pi]^{n}$. Let $\sum_{\bfa\in\Z^n}c_\bfa
\e^{\i\langle\bfa,\bfnu\rangle}$ be its multivariate Fourier
series.  For $j=1,\ldots, n,$ we denote with
$\sum_{a_{j}\in\Z}c_{j,a_{j}}\e^{\i a_{j}\nu_j}$ the Fourier series
of $h_{\alpha_j,\beta_j,\tau }$.  Then, for each $\bfa=(a_1,\ldots,a_n)\in \Z^{n},$
$$c_\bfa=\prod_{j=1}^{n}c_{j,a_j}.
$$
In particular, $c_{\bfzero}=\prod_{j=1}^n\frac{\beta_j-\alpha_j+\tau }{2\pi}$.
The Fourier series of each $h_{\alpha_j,\beta_j,\tau }$ converges
absolutely to this function, and so the same holds for the Fourier series of
$h_{\bfalpha,\bfbeta,\tau}$. Hence, 
$$h_{\bfalpha,\bfbeta,\tau}(\bfnu)
=\sum_{\bfa\in\Z^n}c_\bfa \e^{\i\langle\bfa,\bfnu\rangle}
$$
for $\bfnu\in(-\pi,\pi]^{n}$. 
Applying Lemma \ref{lemm:1}, we obtain 
  \begin{multline} \label{cota1}
 \Sigma_2=\bigg|\frac{1}{D}\sum_{\bfxi}m_{\bfxi}\sum_{\bfa\neq{\bfzero}}c_\bfa \e^{\i\langle\bfa,\arg(\bfxi)\rangle}\bigg| 
 \leq\sum_{\bfa\neq0}|c_\bfa|\bigg|\frac{1}{D}\sum_{\bfxi}m_{\bfxi}\e^{\i\langle\bfa,\arg(\bfxi)\rangle}\bigg|\\
 \leq  2\pi\,\theta(Z) \sum_{\bfa\neq0}|c_\bfa|{\|\bfa\|_2^{\frac12}}
\leq 2^{n+1}\pi\,\theta(Z)\sum_{\bfa\geq{\bfzero}}\bigg(\sum_{s=1}^{n} \sqrt{a_s}\bigg)|c_\bfa|
  \end{multline}
For $s=1$, 
\begin{align*}
\sum_{\bfa\geq{\bfzero}}\sqrt{a_1}|c_\bfa|
=\bigg(\sum_{a_1\geq0}\sqrt{a_1}|c_{1,a_1}|\bigg)\prod_{j=2}^n \bigg(\sum_{a_j\geq0}|c_{j,a_j}|\bigg).
\end{align*}
Using the bounds in~\eqref{eq:7} we get, for $j=2,\ldots,n,$
\begin{align}\label{mijack}
\nonumber \sum_{a_j\geq0}|c_{j,a_j}|&\leq|c_{j,0}|+\sum_{a_j=1}^{3q }
\frac{1}{a_{j}\pi}
+\sum_{a_j=3q +1}^\infty\frac{3}{\pi\tau  a_{j}^2}\\
\nonumber &\leq 
\frac{\beta_{j}-\alpha_{j}+\tau}{2\pi}+ \bigg(1+ \frac{1}{\pi}
\int_{1}^{3q }\frac{dx}{x} \bigg) +
\frac{3}{\pi\tau} \int_{3q }^{\infty}\frac{dx}{x^{2}}
\\
\nonumber &\leq \frac{2\pi}{2\pi}+
1+\frac{1}{\pi}\log\Big(\frac{3}{\tau }\Big)+\frac{3}{\pi\tau } \frac{\tau}{3}\\
&
\leq\frac{1}{\pi}\big(9+\log q \big).
\end{align}
Similarly, we now bound
\begin{align}\label{roboc}
  \nonumber
  \sum_{a_1\geq0}\sqrt{a_1}|c_{1,a_1}|&\leq\sum_{a_1=1}^{3q }\frac{1}{\pi\sqrt{a_1}}
  +\sum_{a_1=3q +1}^\infty\frac{3}{\pi\tau\, a_{1}\sqrt{a_1}}\\
  \nonumber &\le \frac{1}{\pi} \int_{0}^{3q }{x^{-\frac12}}{dx} +
  \frac{3}{\pi\tau} \int_{3q }^{\infty}{x^{-\frac32}}{dx}
  \\
  \nonumber &\leq
  \frac{2}{\pi}\bigg(\frac{\tau}{3}\bigg)^{-\frac12}+\frac{6}{\pi\tau}\bigg(\frac{\tau}{3}\bigg)^{\frac12}\\
  &\le  \frac{4\sqrt{3}}{\pi}{q }^{\frac12}.
\end{align}
It follows from   \eqref{cota1}, \eqref{mijack} and \eqref{roboc} that
\begin{equation}\label{s2}
\Sigma_2\leq n\frac{2^{n+3}\sqrt{3}}{\pi^{n-1}}
  q  ^{\frac12}(9+
  \log({q }))^{n-1}  \, \theta(Z). 
\end{equation}

Next we consider $\Sigma_{3}$.  Set
$\phi(t)=\prod_{j=1}^n\frac{\beta_{j}-\alpha_{j}+\tau t}{2\pi}$ for
$t\in \R$. There exists $0<t_{0}<1$ such that
$\Sigma_{3}=|\phi(1)-\phi(0)| = |\phi'(t_{0})|$ and so
\begin{displaymath}
  \Sigma_{3}\le \sup_{0<t_{0}<1}|\phi'(t_{0})| \le
  \sum_{j=1}^{n}\frac{\tau }{2\pi}\prod_{\ell\neq
    j}\frac{\beta_{\ell}-\alpha_{\ell}+\tau}{2\pi}
\le {n\tau}\frac{(2\pi)^{n-1}}{(2\pi)^{n}}=\frac{n}{2\pi q }.
\end{displaymath}

Finally, we collect (\ref{s1}), (\ref{s2}) and the above inequality to get
\begin{align*}
  \bigg|\frac{\deg(Z_{\bfalpha,\bfbeta})}{\deg(Z)}-\prod_{j=1}^n\frac{\beta_j-\alpha_j}{2\pi}\bigg|
 \leq &  \Sigma_1+\Sigma_2+\Sigma_3\\
\le & 2n\theta(Z) +\frac{n}{\pi q } 
+n\frac{2^{n+3}\sqrt{3}}{\pi^{n-1}}
  q  ^{\frac12}(9+
  \log({q }))^{n-1}  \, \theta(Z) 
+ \frac{n}{2\pi q }\\
= & 2n\theta(Z) +\frac32\frac{n}{\pi q } 
+n\frac{2^{n+3}\sqrt{3}}{\pi^{n-1}}
  q ^{\frac12}(9+
  \log({q }))^{n-1}  \, \theta(Z).
\end{align*}
\end{proof}

\begin{proof}[Proof of Theorem \ref{thm:2}]
For the radius discrepancy, we have
$$
|Z_{\varepsilon}|
= \bigcap_{j=1}^n
 \Big\{\bfxi \in |Z|\mid 1-\varepsilon < |\xi_j| < \frac{1}{1-\varepsilon} \Big\}.
$$
By taking complements in this equality and considering the corresponding
multiplicities, we deduce that
\begin{displaymath}
  \deg(Z) - \deg(Z_{\varepsilon})\le
  \sum_{j=1}^{n}\deg(\chi^{\bfe_{j}}_{*}(Z))- \deg(\chi^{\bfe_{j}}_{*}(Z)_{\varepsilon}). 
\end{displaymath}
Hence, 
\begin{math}
 \Deltar(Z,\varepsilon) \le\sum_{j=1}^n
\Deltar(\chi^{\bfe_{j}}_{*}(Z,\varepsilon)\big)= \rho(Z, \varepsilon),
\end{math}
as stated.

We now consider the bound for the angle discrepancy. For $n=1$,
$\Deltaa(Z)\le \theta(Z)$  and so the
bound in the claim is trivial. Hence, we suppose that $n\geq2$. 
Put then   $\zeta(Z)=
(9 -\log(\theta(Z)))^{\frac23(n-1)}\, \theta(Z)^{2/3}\in \R_{>0}$ for
short  and set 
\begin{displaymath}
  q =\left\lfloor \frac{9 ^{\frac23(n-1)}}{\zeta(Z)}\right\rfloor .
\end{displaymath}
Suppose also that $q \ge 1$. Then
\begin{equation}\label{eq:12}
\frac{9 ^{\frac23(n-1)}}{2\zeta(Z)}<  q  \leq\frac{9 ^{\frac23(n-1)}}{\zeta(Z)}\le\frac{1}{\theta(Z)}.
\end{equation}

Let $\bfalpha=(\alpha_{1},\dots, \alpha_{n})$ and
$\bfbeta=(\beta_{1},\dots, \beta_{n})$ with $-\pi\le
\alpha_{j}<\beta_{j}\le \pi$. Consider first the case where
$\beta_j-\alpha_j\leq\pi$.  In particular,
$\beta_j-\alpha_j+\frac{2}{q }<2\pi$.  Applying Lemma \ref{est33}
and the inequalities \eqref{eq:12}, we deduce that the quantity
$\Big|\frac{\deg(Z_{\bfalpha,\bfbeta})}{\deg(Z)}
-\prod_{j=1}^n\frac{\beta_j-\alpha_j}{2\pi}\Big|$ is bounded from above by
\begin{displaymath}
2n\frac{\zeta(Z)}{9 ^{\frac23(n-1)}} 
+\frac{3\, n}{2\, \pi}\frac{2\, \zeta(Z)}{9 ^{\frac23(n-1)}} 
+n\frac{2^{n+3}\sqrt{3}}{\pi^{n-1}}
  \bigg(\frac{9 ^{\frac23(n-1)}}{\zeta(Z)}\bigg)^{\frac12}(9 -
  \log(\theta(Z)))^{n-1}  \, \theta(Z)  .
\end{displaymath}
Since $n\ge2$, this quantity can be bounded by
\begin{multline*}
n \bigg(\frac{2}{9 ^{\frac23(n-1)}} +\frac{3}{9 ^{\frac23(n-1)}\pi}+\frac{2^{n+3}9 ^{\frac13(n-1)}\sqrt{3}}{\pi^{n-1}}\bigg)
 \zeta(Z)  \\
\le n\bigg(1+\frac{2^{3}\sqrt{3}\pi}{9 ^{\frac13}}\bigg(\frac{2\cdot9 ^{\frac13}}{\pi}\bigg)^{n} \bigg) \zeta(Z) 
\leq 22\,n\left(\frac43\right)^n\zeta(Z),
\end{multline*}
as it can be easily verified that
$\frac{2\cdot9 ^{\frac13}}{\pi}<\frac43$ and
$\frac{2^{3}\sqrt{3}\pi}{9 ^{\frac13}}<21$. 

If $q =0$, then $\frac{9 ^{\frac23(n-1)}}{\zeta(Z)} <1$, which implies
that $ \Deltaa(Z)\le 1\le \zeta(Z)$.  Hence, in the case where $-\pi\le
\alpha_{j}<\beta_{j}\le \pi$ for all $j$, we have
\begin{equation} \label{eq:15}
\Big|\frac{\deg(Z_{\bfalpha,\bfbeta})}{\deg(Z)}
-\prod_{j=1}^n\frac{\beta_j-\alpha_j}{2\pi}\Big|\le 22\,n\left(\frac43\right)^n\zeta(Z). 
\end{equation}

Now, if $\beta_j-\alpha_j>\pi$ for some $j$, we subdivide each of those
intervals $(\alpha_{j},\beta_{j}]$ into two subintervals of length
$\leq \pi$.  We can then decompose $Z_{\bfalpha,\bfbeta}$ as the sum
of at most $2^n$ cycles of the form
$Z_{\bfnu_{0,\sigma},\bfnu_{1,\sigma}}$ where the $j$th coordinate of
$\bfnu_{0,\sigma}$ (respectively $\bfnu_{1,\sigma}$) is either
$\alpha_{j}$ (respectively $\frac{\alpha_{j}+\beta_{j}}{2}$) or
$\frac{\alpha_{j}+\beta_{j}}{2}$ (respectively $\beta_{j}$).
Also, we can expand the product
$$\prod_{j=1}^n\frac{\beta_{j}-\alpha_{j}}{2\pi}
$$
as the sum of the volumes of the sets
$\prod_{j=1}^{m}(\frac{\nu_{0,\sigma,j}}{2\pi},\frac{\bfnu_{1,\sigma,j}}{2\pi}]$.
From here, we easily get that
$$
 \left|\frac{\deg(Z_{\bfalpha,\bfbeta})}{\deg(Z)}-\prod_{j=1}^n\frac{\beta_j-\alpha_j}{2\pi}\right|
  \le \sum_{\sigma}
  \left|\frac{\deg(Z_{\bfnu_{0,\sigma},\bfnu_{1,\sigma}})}{\deg(Z)}-\prod_{j=1}^n\frac{\nu_{1,\sigma,j}-\nu_{0,\sigma,j}}{2\pi}\right|.
$$
The claim follows applying the bound \eqref{eq:15}, which has  to be multiplied by $2^n$. Altogether, we get
$$
\Deltaa(Z)\le 22\,n\bigg(\frac83\bigg)^{n} \zeta(Z),$$
which concludes the proof.
\end{proof}

\section{Bounds for the discrepancy in terms of sparse
  resultants} \label{sec:bounds-discr-terms}

In this section, we consider cycles defined by a system of Laurent
polynomials with given support. We compute their direct image under
monomial projections, in terms of sparse resultants, and we derive Theorem
\ref{thm:1} from the Erd\"os-Tur\'an's
theorem and the results in the previous section. We also establish some
basic properties of the Erd\"os-Tur\'an size. 

\medskip 
We first recall the definition of the sparse
resultant following~\cite{DS:srmet}. 
Let $\cA_{0},\dots,
\cA_{n}$ be a family of $n+1$ non-empty finite subsets of $\Z^{n}$ and put
$\bfcA=(\cA_{0}, \dots, \cA_{n})$. Let $\bfu_{i}=\{u_{i,\bfa}\}_{ a\in
  \cA_{i}}$ be a group of $\# \cA_{i}$ variables, $i=0,\dots,n$, and
set $\bfu=\{\bfu_{0},\dots, \bfu_{n}\}$. For each $i$, let $F_i$ be
the general polynomial with support $\cA_{i}$, that is
\begin{equation}\label{Fi}
  F_{i}=\sum_{\bfa\in \cA_{i}}u_{i,\bfa} \bfx^{\bfa}\in
  \C[\bfu][x_{1}^{\pm1},\dots, x_{n}^{\pm1}],
\end{equation}
and consider the incidence variety
\begin{displaymath}
  W_{\bfcA}=\Big\{(\bfx,\bfu) \in (\C^{\times})^{n}\times
  \prod_{i=0}^{n}\P(\C^{\cA_{i}}) \, \Big|\ 
  F_{i}(\bfu_{i},\bfx)=0\Big\}.
\end{displaymath}
The
direct image of $W_{\bfcA}$ under the projection  \begin{math}
\pi\colon (\C^{\times})^{n}\times \prod_{i=0}^{n}\P(\C^{\cA_{i}})\rightarrow
\prod_{i=0}^{n}\P(\C^{\cA_{i}})
\end{math}
is the Weil divisor of
$\prod_{i=0}^{n}\P(\C^{\cA_{i}})$ given by
\begin{displaymath}
  \pi_{*}(W_{\cA})=
  \begin{cases}
    \deg(\pi|_{W_{\bfcA}}) \ov{\pi(W_{\bfcA})} & \text{ if } \codim(\ov{\pi(W_{\bfcA})}) =1,\\
    0  & \text{ if } \codim(\ov{\pi(W_{\bfcA})}) \ge 2.
  \end{cases}
\end{displaymath}
The \emph{sparse resultant} associated to $\bfcA$, denoted
$\Res_{\bfcA}$, is defined as any primitive equation in $\Z[\bfu]$ of
this Weil divisor. It is well-defined up to a sign. 

According to this definition, sparse resultants are not necessarily
irreducible. If we denote by $\Elim_{\bfcA}$ what
is  classically called the sparse resultant \cite{GKZ94,CLO05,PS93}, then $\Res_{\bfcA}\ne1$ if
and only if $\Elim_{\bfcA}\ne1$ and, if this is the case,
$$
\Res_{\bfcA}=\pm\Elim_{\bfcA}^{\deg(\pi|_{W_{\bfcA}})}.
$$
For instance, for $\cA_{0}=\{0\}, \cA_{1}=\{0,1,2\}\subset \Z$, we
have that
\begin{displaymath}
  \Res_{\bfcA}= \pm
u_{0,0}^{2}, \quad \Elim_{\bfcA}= \pm u_{0,0},
\end{displaymath}
see \cite[Example 3.14]{DS:srmet}.

To recall the basic properties of the sparse resultant that we will need in the
sequel, we need to introduce some definitions. 
Let $H$ be a linear subspace of $\R^{n}$ of dimension $m$ and $P_{i}$,
$i=1,\dots, m$, convex bodies of $H$. The \emph{mixed volume} of these
convex bodies is defined as
 \begin{equation*}
   \MV_{H}(P_{1},\dots,P_{m})=\sum_{j=1}^{m}(-1)^{m-j}\sum_{1\le
     i_{1}<\dots<i_{j}\le m}\vol_{H}(P_{i_{1}}+\dots+P_{i_{j}}),
 \end{equation*}
 where $\vol_{H}$ denotes the Euclidean volume of $H$. We refer to
 \cite{CLO05} for further background on the mixed volume of convex
 bodies.

Write $\C[\bfx^{\pm1}]=\C[x_{1}^{\pm1},\dots, x_{n}^{\pm1}]$ for
short. The \emph{height} of a Laurent polynomial $f=\sum_{\bfa\in
  \Z^{n}}\alpha_{\bfa}\bfx^{\bfa}\in \C[\bfx^{\pm1}]$ is defined as
\begin{displaymath}
\h(f)= \log \big(\max_{\bfa}|\alpha_{\bfa}|\big).    
\end{displaymath}
Given a finite subset $\cB$ of $\Z^{n}$, we denote by $\conv(\cB)$ its
convex hull, which is a lattice polytope of $\R^{n}$.

\begin{proposition} \label{prop:6} Let $\cA_{0},\dots, \cA_{n}\subset
  \Z^{n}$ be a family of $n+1$ non-empty finite subsets and set
  $Q_{i}=\conv(\cA_{i})$.  Then
\begin{equation*}
  \deg_{\bfu_{i}}(\Res_{\bfcA})= \MV_{\R^{n}}(Q_{0}, \dots,
  Q_{i-1},Q_{i+1},\dots, Q_{n}), \quad i=0,\dots, n,
\end{equation*}
and
\begin{displaymath}
\h(\Res_{\bfcA})\le \sum_{i=0}^{n}\MV_{\R^{n}}(Q_{0}, \dots,
  Q_{i-1},Q_{i+1},\dots, Q_{n}) \log (\#\cA_{i}).
\end{displaymath}
\end{proposition}

\begin{proof}
  The formula for the partial degrees is classical, see for
  instance~\cite{GKZ94}. The bound for the height is given by \cite[Theorem
  1]{Som04} for the case where the resultant depends on all the groups
  of variables $\bfu_0,\ldots,\,\bfu_n$. The general case can be found
  in \cite[Proposition 3.15]{DS:srmet}.
\end{proof}

For a family of Laurent polynomials $f_i\in \C[\bfx^{\pm1}]$ with
$\supp(f_{i})\subset \cA_i$, $i=0,\dots,n$, we write
\begin{displaymath}
  \Res_{\bfcA}(f_0,\dots, f_{n}) 
\end{displaymath}
for the evaluation of the resultant at their coefficients. The
following is the multiplicativity formula for sparse resultants.

\begin{proposition}\label{additivity}
  Let $0\le i\le n$ and consider a family of non-empty finite subsets
  $\cA_0,\ldots,\cA_n,\cA_{i}'\subset\Z^{n}$. Let $f_j\in
  \C[\bfx^{\pm1}]$, be a Laurent polynomial with support contained in
  $\cA_j$, $j=0,\dots, n$, and $f_i'\in \C[\bfx^{\pm1}]$ a further Laurent
  polynomial with support contained in $\cA'_i$. Then
\begin{multline*}
\Res_{\cA_{0},\dots, \cA_{i}+\cA_{i}', \dots, \cA_{n}}(f_{0},\dots,
f_{i}f_{i}', \dots, f_{n})\\ =\pm \Res_{\cA_{0},\dots, \cA_{i}, \dots, \cA_{n}}(f_{0},\dots,
f_{i}, \dots, f_{n})\Res_{\cA_{0},\dots, \cA_{i}', \dots, \cA_{n}}(f_{0},\dots,
f_{i}', \dots, f_{n})
\end{multline*}
\end{proposition}

\begin{proof}
  The validity of this formula, with some restrictions, has been
  stablished first in \cite[Proposition 7.1]{PS93}. The general case
  can be found in \cite[Corollary 4.6]{DS:srmet}.
\end{proof}

The \emph{support function} of a compact subset $P\subset \R^n$ is the
function $h_{P}\colon \R^n\to \R$ defined, for $\bfv \in\R^n$, as 
\begin{displaymath}
h_{P}(\bfv )= \inf_{\bfa\in P}\langle \bfa,\bfv \rangle,
\end{displaymath}
where $\langle\cdot,\cdot\rangle$ denotes the inner product of
$\R^{n}$. 
Let $\cB\subset \Z^{n}$ be a finite subset and $f= \sum_{\bfb\in
  \cB}\beta_{\bfb}\bfx^{\bfb}$ a Laurent polynomial with support
contained in $\cB$. 
For $\bfv \in \R^n$, we set
\begin{displaymath}
  \cB^{\bfv }=\{ \bfb\in \cB\mid \langle \bfb,\bfv \rangle =
h_{\conv(\cB)}(\bfv )\}, \quad f^{\bfv }= \sum_{\bfb\in
  \cB^{\bfv }}\beta_{\bfb}x^{\bfb}. 
\end{displaymath}

\begin{definition} \label{def:1} Let $\cA_{1},\dots,
  \cA_{n}\subset\Z^{n}$ be a family of $n$ non-empty finite subsets,
  $\bfv \in\Z^n\setminus \{0\}$, and $\bfv ^{\bot}\subset\R^{n}$ the
  orthogonal subspace. Then $\Z^n\cap \bfv ^{\bot}$ is a lattice of
  rank $n-1$ and, for $i=1,\dots,n$, there exists $\bfb _{i,\bfv }\in
  \Z^n$ such that $\cA_{i}^{\bfv }-\bfb_{i,\bfv } \subset \Z^n\cap
  \bfv ^{\bot}$.  The \emph{resultant of $\cA_{1},\dots, \cA_{n}$ in
    the direction of $\bfv $}, denoted $\Res_{\cA_{1}^{\bfv
    },\dots,\cA_{n}^{\bfv }}$, is defined as the resultant of the
  family of finite subsets $\cA_{i}^{\bfv }-\bfb_{i,\bfv }\subset
  \Z^n\cap \bfv ^{\bot}$.
  
  Let $f_{i}\in \C[x_1^{\pm1},\ldots, x_n^{\pm1}]$, $i=1,\dots,n$,
  with $\supp(f_{i})\subset\cA_{i}$.  For each $i$, write
  $f_{i}^{\bfv }=\bfx^{\bfb_{i,\bfv }}g_{i,\bfv }$ for a Laurent polynomial $g_{i,\bfv }\in
  \C[\Z^n\cap \bfv ^{\bot}]\simeq\C[y_1^{\pm1},\ldots,y_{n-1}^{\pm1}]$
  with $\supp(g_{i,\bfv })\subset \cA_{i}^{\bfv }-\bfb_{i,\bfv }$.  The expression
\begin{displaymath}
  \Res_{\cA_{1}^{\bfv },\dots,\cA_{n}^{\bfv }}(f_{1}^{\bfv },\dots,f_{n}^{\bfv })
\end{displaymath}
is defined as the evaluation of this directional resultant at the
coefficients of the $g_{i,\bfv }$.  These constructions are
independent of the choice of the vectors $\bfb_{i,\bfv }$.
\end{definition}

We have that
$\Res_{\cA_{1}^{\bfv },\dots,\cA_{n}^{\bfv }}\ne 1$ only if $\bfv $ is
an inner normal to a facet of the Minkowski sum $\sum_{i=1}^{n}\conv(\cA_{i})$.
In particular, the number of non-trivial directional resultants of the
family $\cA_{1},\dots, \cA_{n}$ is finite.

With notation as in Definition \ref{def:1},  write  $\bff=(f_{1},\dots,
f_{n})$ for short. We denote by $V(\bff)_{0}\subset \big(\C^\times\big)^n$ the set of isolated
solutions in the algebraic torus of the system of equations $f_{1}=\dots=
f_{n}=0$ and we set 
\begin{displaymath}
  Z(f_{1},\dots,
f_{n}) = \sum_{\bfxi\in V(\bff)_{0}} \mult(\bfxi|\bff) [\bfxi]
\end{displaymath}
for  the associated 0-dimensional cycle, where $ \mult(\bfxi|\bff)$ denotes
the intersection multiplicity of $\bff$ at a point $\bfxi$.  
For $f_{0}\in \C[\bfx^{\pm1}]$, we set
\begin{displaymath}
  f_{0}(Z(f_{1},\dots,
f_{n}))= \prod_{\bfxi}
f_{0}(\bfxi)^{ \mult(\bfxi|\bff)}.
\end{displaymath}

The following result is known as the Poisson formula for sparse resultants. 

\begin{proposition} \label{prop:5} Let $\bfcA=(\cA_0,\ldots,\cA_n)$ be
  a family of non-empty finite subsets of $\Z^n$ and $f_i\in \C[\bfx^{\pm1}]$ a
  Laurent polynomial with $\supp(f_{i})\subset \cA_i$, $i=0,\dots,n$. Suppose that
  $\Res_{\cA_{1}^{\bfv },\dots,\cA_{n}^{\bfv }}(f_{1}^{\bfv },\dots,f_{n}^{\bfv })\ne
  0$ for all $\bfv \in \Z^n\setminus \{\bfzero\}$. Then
\begin{equation*}
  \Res_{\bfcA}(f_{0},f_{1},\dots,f_{n})= \pm
\bigg(\prod_{\bfv }\Res_{\cA_{1}^{\bfv
    },\dots,\cA_{n}^{\bfv }}(f_{1}^{\bfv },\dots,f_{n}^{\bfv
  })^{-h_{\cA_{0}}(\bfv )} \bigg)  f_{0}(Z(f_{1},\dots,f_{n}))  ,
\end{equation*}
the product being over all primitive elements $\bfv \in\Z^n$.

\end{proposition}

\begin{proof}
  This formula has been obtained, under some restrictions on the
  supports, by Pedersen and Sturmfels in \cite[Theorem 1.1]{PS93}. The
  general case can be found in~\cite[Theorem 1.1]{DS:srmet}.
\end{proof}

From now on, we fix a family of non-empty finite subsets $\cA_{1},\dots,
\cA_{n}$ of $\Z^{n}$ such that
$\MV_{\R^{n}}(Q_{1},\dots,Q_{n})\ge1$, where $Q_{i}=\conv(\cA_{i})$.  We consider also a family of Laurent
polynomials $f_{1},\dots, f_{n}\in \C[\bfx^{\pm1}]$ with
$\supp(f_{i})\subset \cA_{i}$ and
$\Res_{\cA_{1}^{\bfv},\dots,\cA_{n}^{\bfv}}(f_{1}^{\bfv},\dots,f_{n}^{\bfv})\ne
0$ for all $\bfv\in \Z^{n}\setminus \{0\}$.  By
Bernstein's theorem \cite[Theorem~B]{Ber75},
\begin{equation*} 
  \deg(Z(\bff))= \MV_{\R^{n}}(Q_{1},\dots, Q_{n})\ge 1.
\end{equation*}
Consider the projection  $ \pi_{\bfe_{1}}\colon \R^{n}\to \R^{n-1}$ 
given by $\pi_{\bfe_{1}}(x_1,x_{2},\dots, x_n)=(x_2,\dots,x_{n})$.  If we
regard each $f_i$ as a Laurent polynomial in the group of variables $\bfx' :=
(x_2,\dots,x_{n}) $ with coefficients in the ring $\C[{x_1}^{\pm1}]$,
its support with respect to $\bfx'$ is contained in the
finite subset $\pi_{\bfe_1}(\cA_i)$ of $\R^{n-1}$.  Set then
$$
R(\bff)= \Res_{\pi_{e_1}(\cA_1),\dots,\pi_{e_1}(\cA_{n})}
\big(f_1(x_1,\bfx') ,\dots, f_n(x_1,\bfx')\big) \in \C[x_1^{\pm1}]
$$
for the evaluation of the resultant at these coefficients in
$\C[x_{1}^{\pm1}]$.

Recall that the {sup-norm} of a Laurent polynomial $f\in
\C[\bfx^{\pm1}]$ is defined as
\begin{math}
 \|f\|_{\sup}=\sup_{\bfw\in (S^1)^n}|f(\bfw)|.    
\end{math}
In general, it holds that 
\begin{equation} \label{eq:17}
\h(f)\le \log \|f\|_{\sup}
\end{equation}
This is a consequence of Cauchy's formula for the coefficients of the
Laurent  expansion of a holomorphic function on $(\C^{\times})^{N}$ (see for
instance \cite[page 1255]{Som04}).

The following result gives a bound for the sup-norm of $R(\bff)$. Its
proof is a variant of that for \cite[Lemma~1.3]{Som04}.

\begin{lemma}\label{bound}
Let notation be as above. Then
$$
\log\|R(\bff)\|_{\sup}\le \sum_{j=1}^n\MV_{\bfe_1^\bot}(\{\pi_{\bfe_1}(Q_\ell)\}_{\ell\ne j})
\log\|f_j\|_{\sup}.
$$
\end{lemma}

\begin{proof}
Let $k\ge 1$. Proposition \ref{additivity} implies that
\begin{equation}\label{eq:19}
R(\bff)^{k^{n}}=\Res_{k\pi_{\bfe_1}(\cA_1),\dots,k\pi_{\bfe_1}(\cA_{n})}
(f_1(x_1,\bfx')^k ,\dots, f_n(x_1,\bfx')^k),
\end{equation}
where $k\pi_{\bfe_1}(\cA_j)$ denotes the pointwise sum of $k$ copies
of $\pi_{\bfe_1}(\cA_j)$.  For short, set
$R_{k}=\Res_{k\pi_{\bfe_1}(\cA_1),\dots,k\pi_{\bfe_1}(\cA_{n})}$ and
$\bff^{k}=(f_{1}^{k},\dots,f_{n}^{k})$, so that the identity above can
be rewritten as $ R(\bff)^{k^{n}}=R_{k}(\bff^{k})$. By Proposition
\ref{prop:6},  the partial degrees of this
resultant are given by
\begin{displaymath}
  \deg_{\bfu_{j}}(R_{k})=\MV_{\bfe_1^\bot}(\{k\pi_{\bfe_1}(Q_\ell)\}_{\ell\ne
    j})=k^{n-1}\MV_{\bfe_1^\bot}(\{\pi_{\bfe_1}(Q_\ell)\}_{\ell\ne j}), 
\end{displaymath}
where $\bfu_{j}$ is a group of $\#k\cA_{j}$ variables, for $j=1,\dots,
n$. In particular, the logarithm of its number of monomials is bounded
from above as
\begin{multline*}
  \log(\#\supp(R_{k}))\le\log\bigg(  \prod_{j=1}^{n}{\#k\cA_{j}+
    \deg_{\bfu_{j}}(R_{k})\choose \#k\cA_{j}}\bigg) \\ \le
    \sum_{j=1}^{n}\deg_{\bfu_{j}}(R_{k})\log(\#k\cA_{j}+1)= O(k^{n-1}\log(k+1)),
\end{multline*}
since $\#k\cA_{j}\le (k+1)^{c_{1}n}$ for a constant $c_{1}$
independent of $k$.  By Proposition \ref{prop:6}, the height of this
resultant is bounded from above by
\begin{displaymath}
\h(R_{k})\le \sum_{j=1}^{n}\deg_{\bfu_{j}}(R_{k})\log(\#k\cA_{j})= O(k^{n-1}\log(k+1)).
\end{displaymath}
Let $w_{1}\in S^{1}$. Using~\eqref{eq:19}, \eqref{eq:17}, the
previous bounds, and the fact that $\|f_{j}^{k}\|_{\sup}=\|f_{j}\|^{k}_{\sup}$, we deduce that
\begin{align*}
 {k^{n}}\log\|R(\bff)\|_{\sup}&= \log\|R_{k}(\bff^{k}) \|_{\sup}\\
&\le
  \h(R_{k}) + \sum_{j=1}^{n}\deg_{\bfu_{j}}(R_{j})
  \log\|f_{j}^{k}\|_{\sup} + \log(\#\supp(R_{k})) \\ 
&= k^{n} \Big(\sum_{j=1}^n\MV_{\bfe_1^\bot}(\{\pi_{\bfe_1}(Q_\ell)\}_{\ell\ne j})
\log\|f_j\|_{\sup}\Big) + O(k^{n-1}\log(k+1)). 
\end{align*}
The result then follows by dividing both sides of this inequality by
$k^n$ and letting $k\to\infty$.
\end{proof}

For $\bfa\in\Z^n\setminus\{0\}$ and $z$ an additional variable, set 
$$
E_{\bfa}(\bff)= \Res_{\{\bfzero, \bfa\}, \cA_{1}, \dots,
        \cA_{n}}(z-\bfx^\bfa,f_1,\ldots,f_n) \in \C[z].
$$ 
Due to the Poisson formula for sparse resultants given in Proposition
\ref{prop:5}, we have that $Z(E_{\bfa}(\bff))=\chi_*^\bfa(Z(\bff)),$
and so $E_{\bfa}(\bff)$ can be regarded as an elimination polynomial
for the cycle $Z(\bff)$ with respect to the monomial projection
$\chi^{\bfa}$.

By \cite[Theorem 1.4]{DS:srmet}, there exists $m\in \Z$ such that
\begin{equation} \label{eq:6}
E_{\bfe_1}(\bff)(x_1) =x_{1}^{m}R(\bff).   
\end{equation}
Hence, Lemma \ref{bound}
can be regarded as a bound for the sup-norm of the elimination
polynomial $E_{\bfe_{1}}(\bff)$.  Our next step is to extend this
result to an arbitrary $\bfa$. Recall that $\pi_\bfa\colon \R^n\to
\bfa^\bot$ denotes the orthogonal projection onto the hyperplane
$\bfa^{\bot}\subset\R^{n}$.

\begin{lemma}\label{prop:1}
Following the notation above,
$$
\log\|E_{\bfa}(\bff)\|_{\sup}\le \|\bfa\|_{2} \sum_{j=1}^n
\MV_{\bfa^\bot}(\{\pi_\bfa(Q_\ell)\}_{\ell\ne j})
\log\|f_j\|_{\sup}.
$$
\end{lemma}

\begin{proof}
  Consider first the case where $\bfa\in\Z^n$ is primitive. The
  quotient $\Z^n/ \bfa\Z$ is torsion-free and so $\bfa$ can be
  completed to a basis of $\Z^n$. Equivalently, there is an invertible
  matrix $A\in \SL_n(\Z)$ with first row $\bfa$.  Set $\bfa, \bfa_2,
  \dots, \bfa_n$ and $\bfb_1,\bfb_2, \dots, \bfb_n$ for the rows of
  $A$ and of $A^{-1}$, respectively. There is a commutative diagram
\begin{displaymath}
\xymatrix{
(\C^\times)^n
\ar[r]^{\chi^\bfa} \ar[d]^{\varphi_A} & \C^\times \\
(\C^\times)^n  \ar@/^1pc/[u]^{\varphi_{A^{-1}}} \ar[ur]_{\chi^{\bfe_{1}}} &}
\end{displaymath}
where $\varphi_A$ and  $\varphi_{A^{-1}}$ are the monomial
isomorphisms given by $\bfx\mapsto \bfx^A= (\bfx^{\bfa_1},\dots, \bfx^{\bfa_n})$ 
and $\bfx\mapsto \bfx^{A^{-1}}= (\bfx^{\bfb_1},\dots, \bfx^{\bfb_n})$,
respectively. 
Let $\bfy=(y_1,\dots, y_n)$ denote the coordinates of the algebraic torus below. For $\ell=1,\ldots, n,$ set
$$
f_\ell^A(\bfy)= \varphi_{A^{-1}}^{*}f_\ell=f_\ell(\bfy^{\bfb_1},\dots,
\bfy^{\bfb_n}) \in \C[\bfy^{\pm1}], 
$$
so that $(\varphi_A)_{*}Z(\bff)=Z(\bff^A)$. Hence,
$E_{\bfa}(\bff)=E_{\bff^A, \bfe_1}$ and so Lemma~\ref{bound} combined
with \eqref{eq:6} implies that
\begin{equation}\label{eq:10}
\log\|E_{\bfa}(\bff)\|_{\sup}\le \sum_{j=1}^n \MV_{\bfe_1^\bot}(\{\pi_{\bfe_{1}}(\newton(f_\ell^A))\}_{\ell\ne j})
\log\|f_j^A\|_{\sup},
\end{equation}
where $\newton(f_\ell^A)$ is the Newton polytope of $f_{\ell^{A}}$. We
have 
\begin{math}
\pi_{\bfe_{1}}(\newton(f_\ell^A))= \wt B (\newton(f_\ell)) \subset \wt B(Q_{\ell})  
\end{math}
for the linear map $\wt B\colon  \R^n\to \R^{n-1}$ given by $\wt
B(\bfx)=(\langle \bfx,\bfb_2\rangle, \dots, \langle
\bfx,\bfb_n\rangle)$. Let $\{\bfv_2,\ldots,\bfv_n\}$ be an orthonormal
basis of $\bfa^\bot$ and consider a second commutative diagram
\begin{displaymath}
\xymatrix{
\R^n
\ar[r]^{\wt B} \ar[d]_{\pi_\bfa} & \R^{n-1}\ar[d]^C \\
\bfa^\bot \ar[r]_U & \bfa^\bot}
\end{displaymath}
where $C$ is the linear map defined by $ C (y_2,\dots, y_n)=y_2\bfv_2
+ \cdots + y_n \bfv_n$. It is easy to verify that $U\in
\GL(\bfa^\bot)$ is uniquely determined by $\pi_\bfa$, $\wt{B}$ and
$C$.  Since $C$ maps the canonical basis of $\R^{n-1}$
into an orthonormal basis of $\bfa^\bot$, it is an isometry between
these two spaces. On the other hand, a straightforward computation
shows that
$$U(\bfv_j)=\sum_{k=2}^n\langle \bfv_j,\bfb_k\rangle \bfv_k=\bfb_j, \ \quad j=2,\ldots, n.
$$
We note that $\bfb_2,\dots, \bfb_n$ is a basis of the $\Z$-module
$\bfa^\bot\cap \Z^n$.  The Brill-Gordan formula~\cite[Chapitre~3,
\S~11, Proposition~15]{Bourbaki:ema} implies that $\vol(
\bfa^\bot/(\bfa^\bot \cap \Z^n))=\|\bfa\|_{2}$. Hence,
\begin{equation}
  \label{eq:16}
|\det(U)|=\vol( \bfa^\bot/(\bfa^\bot \cap \Z^n))=\|\bfa\|_{2}.  
\end{equation}

Since $C$ is an isometry, \cite[Theorem 4.12(a)]{CLO05} implies that, for $j=1,\dots,n$,
\begin{equation*}
  \MV_{\bfe_1^\bot}(\{\wt B(Q_\ell)\}_{\ell\ne j}) =
  \MV_{\bfa^\bot}(\{C\circ \wt B(Q_\ell)\}_{\ell\ne j})
  = \MV_{\bfa^\bot}(\{U\circ\pi_\bfa (Q_\ell)\}_{\ell\ne j}).
\end{equation*}
By \eqref{eq:16}, $\|\bfa\|^{-1/(n-1)}U$ is a volume preserving
map. Applying \cite[Theorem~4.12(a,b)]{CLO05}, we deduce that
\begin{equation*}
   \MV_{\bfa^\bot}(\{U\circ\pi_\bfa (Q_\ell)\}_{\ell\ne j})
  =\|\bfa\|_{2} \MV_{\bfa^\bot}(\{\pi_\bfa (Q_\ell)\}_{\ell\ne j}).
\end{equation*}
In addition,  $\varphi_A$ gives an automorphism
of $(S^1)^n$ and so $\|f_\ell^A\|_{\sup}=\|f_\ell\|_{\sup}$. We
conclude that, when $\bfa$ is primitive, 
\begin{equation}\label{eq:11}
\log\|E_{\bfa}(\bff)\|_{\sup}\le \|\bfa\|_{2} \sum_{j=1}^n
\MV_{\bfa^\bot}(\{\pi_\bfa(Q_\ell)\}_{\ell\ne j})
\log\|f_j\|_{\sup}.
\end{equation}

Now let $\bfa\in \Z^n\setminus \{0\}$ be any vector. Choose
a primitive $\bfa'\in \Z^{n}$ and $m\in \Z_{\ge1} $ such that $\bfa= m \bfa'$.
Using Proposition \ref{additivity}, we deduce that
$$
E_{\bfa}(\bff) (z)=\pm  \prod_{\omega\in \mu_m} E_{\bfa'}(\bff)(\omega z)
$$
where $\mu_m$ denotes the set of $m$-th roots of $1$.  Hence
$\|\bfa\|_{2}=m\|\bfa'\|_{2} $, $\pi_\bfa=\pi_{\bfa'}$ and
$\log\|E_{\bfa}(\bff)\|_{\sup}\le m \log\|E_{\bfa'}(\bff)\|_{\sup}$.
The result follows from the bound~(\ref{eq:11}) applied to $\bfa'$.
\end{proof}

\begin{proof}[Proof of Theorem \ref{thm:1}]
Let $\bfa\in \Z^{n}\setminus \{\bfzero\}$. Applying
  Proposition \ref{prop:5} with $f_0=z-\bfx^\bfa$ and
  $f_0=z-\bfx^{-\bfa},$ we get that the product of the leading and the
  constant coefficients of $E_{\bfa}(\bff)$ is equal to
$$\pm \prod_{\bfv}\Res_{\cA_{1}^{\bfv},\dots,\cA_{n}^{\bfv}}(f_{1}^{\bfv},\dots,f_{n}^{\bfv})^{|\langle
    \bfv,\bfa\rangle|}.
$$
Recall also that $Z(E_{\bfa}(\bff)(z))=\bfchi_*^\bfa\big(Z(\bff)\big).$
The Erd\"os-Tur\'an's theorem (Theorem \ref{mtunivariate}) then implies that
$$\Deltaa\big(\bfchi_*^\bfa(Z(\bff))\big)\leq c\, \sqrt{\frac{1}{D}\log\Bigg(\frac{\|E_{\bfa}(\bff)(z)\|_{\sup}}
{\prod_{\bfv}|\Res_{\cA_{1}^{\bfv},\dots,\cA_{n}^{\bfv}}(f_{1}^{\bfv},\dots,f_{n}^{\bfv})|^{\frac{|\langle\bfv,\bfa\rangle |}{2}}}\Bigg)},
$$
with $c=2.5619\dots$
 Lemma \ref{prop:1} implies that $ \log\bigg(\frac{\|E_{\bfa}(\bff)(z)\|_{\sup}}
{\prod_{\bfv}|\Res_{\cA_{1}^{\bfv},\dots,\cA_{n}^{\bfv}}(f_{1}^{\bfv},\dots,f_{n}^{\bfv})|^{\frac{|\langle\bfv,\bfa\rangle
      |}{2}}}\bigg)$ is bounded from above by the quantity
 \begin{displaymath}
{\|\bfa\|_{2}} \log\Bigg(\frac{\prod_{i=1}^{n}\|f_i\|_{\sup}^{D_{\bfw,i}}} 
 { \prod_{\bfv}|\Res_{\cA_{1}^{\bfv},\dots,\cA_{n}^{\bfv}}(f_{1}^{\bfv},\dots,f_{n}^{\bfv})|^{\frac{|\langle
          \bfv,\bfw \rangle|}{2}}}\Bigg),
 \end{displaymath}
for $\bfw=\frac{\bfa}{\|\bfa\|_{2}}\in S^{n-1}$. 
From the definitions of $\theta(Z(\bff))$ and of the Erd\"os-Tur\'an
size~$\eta(\bff)$ given in \eqref{theta} and \eqref{eq:1}, respectively, we get
$$\theta(Z(\bff))\leq \min\{1,c\, \sqrt{\eta(\bff)}\}.
$$
Applying Theorem \ref{thm:2} and the fact that the function
$ t^{\frac23}(9-\log(t))$ is monotonically
increasing in the interval $(0,1]$, we deduce that
\begin{align*}
\Deltaa(Z(\bff))&\le 22 n\bigg(\frac83\bigg)^{n} \big(9-\log(\min\{1,c\,
\sqrt{\eta(\bff)}\})\big)^{\frac23(n-1)} \min\{1,c\,
\sqrt{\eta(\bff)}\}^{\frac23}\\
&\le 22  n\bigg(\frac83\bigg)^{n}
2^{-\frac23(n-1)}(18+\log^{+}({\eta(\bff)}^{-1}))^{\frac23(n-1)}
c^{\frac23}  {\eta(\bff)}^{\frac13}\\
&\le
66\,n\, 2^{n} (18+\log^{+}({\eta(\bff)}^{-1}))^{\frac23(n-1)}
 {\eta(\bff)}^{\frac13},
\end{align*}
which gives the bound for the angle discrepancy.  For the radius
discrepancy, we use the bounds given in Theorem \ref{mtunivariate},
\eqref{theta}, and Theorem \ref{thm:2} to get, for $0<\varepsilon<1,$
\begin{displaymath}
\Deltar(\bff,\varepsilon)\leq\sum_{j=1}^n\Deltar(\bfchi_*^{\bfe_j}(Z(\bff)),\varepsilon)\\
\leq \frac{2n}{\varepsilon}\,\eta(\bff). 
\end{displaymath}
This concludes with the proof of the theorem.
\end{proof}

We next study a  number of basic properties of the Erd\"os-Tur\'an
size. The following proposition shows that this notion generalizes the
measure of polynomials that appears in the statement of Theorem
\ref{mtunivariate}.

\begin{proposition} \label{prop:2}
Let $d\ge1$ and 
$f= a_0+\cdots+a_dx^d  \in \C[x] $
with $a_0a_d \ne 0$. Then  
$$
\eta(f)=\frac{1}{d}\log\bigg(\frac{\|f\|_{\sup}}{\sqrt{|a_0a_d|}}\bigg).$$
\end{proposition}

\begin{proof}
The directional resultants of $f$ are 
\begin{displaymath}
  \Res_{\bfv}(f^{\bfv})
  \begin{cases}
    \pm a_{0}& \text{ for } \bfv=1,\\
\pm a_{d}& \text{ for } \bfv=-1. 
  \end{cases}
\end{displaymath}
Moreover, $D=\MV_{\R}([0,d])=d$, $D_{\bfw,1}=1$ for $\bfw\in
S^{0}=\{\pm1\}$, and $|\langle \bfv,\bfw\rangle|=1$ for all
$\bfv,\bfw\in \{\pm1\}$. The formula for $\eta(f)$ then boils down to
$\frac{1}{d}\log\Big(\frac{\|f\|_{\sup}}{\sqrt{|a_0a_d|}}\Big).$
\end{proof}

We denote by $\Delta^{n}=\conv(\bfzero, \bfe_1,\ldots,\bfe_n)\subset
\R^{n}$ the
standard $n$-simplex.

\begin{proposition} \label{prop:3} Let $\cA_{1},\dots,\cA_{n}$ be a
  family of non-empty finite subsets of $ \Z^{n} $ such that
  $\MV_{\R^{n}}(Q_{1},\dots,Q_{n})\ge1$ with $Q_{i}=\conv(\cA_{i})$.
  Let $f_{1},\dots, f_{n}\in \C[x_{1}^{\pm1},\dots, x_{n}^{\pm1}]$
  with $\supp(f_{i})\subset \cA_{i}$ and such that
  $\Res_{\cA_{1}^{\bfv},\dots,\cA_{n}^{\bfv}}(f_{1}^{\bfv},\dots,f_{n}^{\bfv})\ne
  0$ for all $\bfv\in \Z^{n}\setminus \{0\}$.
\begin{enumerate}
\item \label{item:6} $\eta(\bff)<\infty$.
\item \label{item:7} Let $\gamma_{1},\dots, \gamma_{n}\in
  \C^{\times}$. Then $\eta(\gamma_{1} f_{1},\dots,\gamma_{n}f_{n})=\eta(f_{1},\dots,f_{n})$.
\item \label{item:8} Let $d_j\in\Z_{\ge1}$ and $\bfb_j\in\Z^n$ such that
  $Q_{j}\subset d_j\Delta^{n}+\bfb_j$, $j=1,\dots,n$. Then
  \begin{multline*}
\eta(\bff)\le
\frac{1}{\MV_{\R^{n}}(Q_{1},\dots,Q_{n})}\bigg((n+\sqrt{n})\bigg(\prod_{j=1}^{n}d_j\bigg)\sum_{j=1}^n\frac{\log\|f_j\|_{\sup}}{d_j}\\
+      \sum_{\bfv}\frac{{\|\bfv\|_{2}}}{2}\log^+|\Res_{\cA_{1}^{\bfv},\dots,\cA_{n}^{\bfv}}(f_{1}^{\bfv},\dots,f_{n}^{\bfv})^{-1}|\bigg),
  \end{multline*} 
  the second sum being taken over all primitive vectors
  $\bfv\in\Z^n$. Moreover, if   $f_{1},\dots, f_{n}\in
  \Z[x_{1}^{\pm1},\dots, x_{n}^{\pm1}]$,  then 
  \begin{displaymath}
    \eta(\bff)\le
\frac{(n+\sqrt{n})\big(\prod_{j=1}^{n}d_j\big)}{\MV_{\R^{n}}(Q_{1},\dots,Q_{n})} \sum_{j=1}^n\frac{\log\|f_j\|_{\sup}}{d_j}.
  \end{displaymath}
\end{enumerate}
\end{proposition}

\begin{proof}
The statement of  \eqref{item:6} is clear, since $\eta(\bff)$ is defined as the
  supremum of a continuous function over the compact set~$S^{n-1}$. 

For  \eqref{item:7}, let $\bfa\in \Z^{n}\setminus \{\bfzero\}$. 
As explained in the proof of Theorem \ref{thm:1}, the product of the leading and the
  constant coefficients of $E_{\bfa}(\bff)$ is equal to
$$\pm \prod_{\bfv}\Res_{\cA_{1}^{\bfv},\dots,\cA_{n}^{\bfv}}(f_{1}^{\bfv},\dots,f_{n}^{\bfv})^{|\langle
    \bfv,\bfa\rangle|}.
$$
Hence, the denominator in the definition of $\eta(\bff)$ is
multihomogeneous in the coefficients of each $f_{j}$, of partial
degrees equal to  ${\|\bfa\|_{2}}^{-1}$ times those of
  $E_{\bfa}(\bff)$. 
Hence, 
\begin{displaymath}
\frac{1}{\|\bfa\|_{2}}\deg_{f_{j}}(E_{\bfa}(\bff))= 
\MV_{\bfa^\bot}\big(\pi_\bfa(Q_1),\dots,
\pi_\bfw(Q_{j-1}),\pi_\bfa(Q_{j+1}),\dots,\pi_\bfa(Q_n))=D_{\bfw,j}
\end{displaymath}
for $\bfw=\frac{\bfa}{\|\bfa\|_2}$, which  implies the statement.

For \eqref{item:8}, let $\bfw\in S^{n-1}$. Then
$\pi_\bfw(Q_j)\subset\pi_\bfw(d_j \Delta^{n}+\bfb_j)$. Due to
the monotonicity of the mixed volume with respect to the inclusion,
plus its properties of homogeneity and invariance under translation,
we deduce that, for $j=1,\dots,n$,
\begin{multline}\label{huno}
\MV_{\bfw^{\bot}}(\{\pi_\bfw(Q_\ell)\}_{\ell\ne
  j})\leq\MV_{\bfw^{\bot}}(\{\pi_\bfw(d_\ell\Delta^{n}+\bfb_\ell)\}_{\ell\ne
j})\\ \leq(n-1)!\Big(\prod_{\ell\ne j} d_{\ell}\Big)\vol_{\bfw^{\bot}}(\pi_\bfw(\Delta^{n})) .
\end{multline}
The projected simplex $\pi_\bfw(\Delta^{n})$ can be covered by the union
of the projection of its facets. One of the facets of $\Delta^{n}$ has
$(n-1)$-dimensional volume equal to $\frac{\sqrt{n}}{(n-1)!}$, while 
the other $n$ facets have $(n-1)$-dimensional  volume equal to
$\frac{1}{(n-1)!}.$ Since  the volume cannot increase under orthogonal
projections, we have that
\begin{equation}\label{hdos}
\vol_{\bfw^{\bot}}(\pi_\bfw(\Delta^{n}))\leq \frac{\sqrt{n}+n}{(n-1)!}.
\end{equation}
In addition, $|\langle \bfv,\bfw\rangle|\le \|\bfv\|_{2}$ since
$\bfw\in S^{n-1}$. Then, the first part of the statement follows from
\eqref{huno},\,\eqref{hdos} and the definition of $\eta(\bff)$.  The
second part follows from the fact that, if the coefficients of the
$f_{i}$'s are integers, then the relevant directional resultants are
nonzero integers and so their absolute values are at least 1.
\end{proof}

Let us consider the statement of Proposition \ref{prop:3}\eqref{item:8},
in the classical dense case $Q_j=d_j\Delta^{n}$ for all $j$. In this situation, the
only primitive vectors $\bfv$ to consider are $\bfv=\bfe_{i}$,
$i=1,\dots,n$, and $\bfv=\bfe_{0}:=- \sum_{i=1}^{n}\bfe_{i}.$ Given
$d_{j}\ge1$, $j=1,\dots,n$, we write $ \Res_{d_{1},\dots,d_{n}} $ for
the resultant of $n$ homogeneous polynomials in $n$ variables of respective degrees
$d_{1},\dots,d_{n},$ as defined in \cite{CLO05}.  Given a system of polynomials $f_1,\ldots,
f_n\in\C[x_1,\ldots, x_n]$ with $\deg(f_{j})\le d_{j}$ and
$i=0,\dots,n$, the initial polynomials
$f_{1}^{\bfe_i},\dots,f_{n}^{\bfe_i}$ form a system of $n$ polynomials
of degrees $d_{1},\dots,d_{n}$. In particular, we can evaluate
$\Res_{d_{1},\dots,d_{n}}$ at these polynomials. If we assume that
these resultants are nonzero, we obtain
\begin{multline*}
  \eta(\bff)\leq(n+\sqrt{n})\sum_{j=1}^n\frac{\log\|f_j\|_{\sup}}{d_j}
  +\frac{1}{2\prod_{j=1}^{n}d_j}\bigg(
  \sqrt{n}\log^+|\Res_{d_{1},\dots,d_{n}}(f_{1}^{\bfe_{0}},\dots,f_{n}^{\bfe_{0}})^{-1}|
  \\+
  \sum_{i=1}^n\log^+|\Res_{d_{1},\dots,d_{n}}(f_{1}^{\bfe_i},\dots,f_{n}^{\bfe_i})^{-1}|\bigg).
\end{multline*}
In particular, if $f_1,\ldots, f_n\in\Z[x_1,\ldots, x_n]$, then
\begin{equation*}
 \eta(\bff)\le (n+\sqrt{n})\sum_{j=1}^n\frac{\log\|f_j\|_{\sup}}{d_j}.
\end{equation*}

\section{Asymptotic equidistribution} \label{sec:equid-prob}

We will apply here the results in the previous sections to study the
asymptotic distribution of the roots of a sequence of systems of
Laurent polynomials over $\Z$ and of random systems of Laurent
polynomials over~$\C$.

First, we will consider  polynomials over $\Z$.
Let $Q_i,\dots, Q_{n}\subset \R^{n}$ be a family of lattice polytopes
such that $\MV_{\R^{n}}(Q_1,\ldots, Q_n)\ge1$.  For each integer $\kappa\geq1$ and
$i=1,\dots, n$, consider the finite subset of $\Z^{n}$ given by
\begin{equation}
  \label{eq:4}
\cA_{\kappa,i}=\kappa Q_i\cap\Z^n .
\end{equation}

\begin{proposition} \label{prop:7} For  $\kappa\ge1 $ let 
  $\bff_{\kappa}=(f_{\kappa,1},\dots, f_{\kappa,n}) $ be a family of Laurent
  polynomials in $\Z[x_1^{\pm1},\dots, x_n^{\pm1}] $ such that $
  \supp(f_{\kappa,i})\subset \kappa Q_{i}$ and
  $\Res_{\cA_{\kappa,1}^{\bfv},\dots,\cA_{\kappa,n}^{\bfv}}(f_{\kappa,1}^{\bfv},\ldots,f_{\kappa,n}^{\bfv})\ne0$
  for all $\bfv\in \Z^{n}\setminus \{\bfzero\}$.  Then there is a
  constant $c_{1} >0$ which does not depend on $\kappa$ such that
$$
\Deltaa(Z(\bff_\kappa))\le c_{1} \,
\bigg(\frac{\sum_{i=1}^{n}\log\|f_{\kappa,i}\|_{\sup}}{\kappa}\bigg)^{\frac13} 
 \bigg(1+\log^{+}\bigg(\frac{\kappa}{\sum_{i=1}^{n}\log\|f_{\kappa,i}\|_{\sup}}\bigg)\bigg)^{\frac{2}{3}(n-1)}
$$
and, for any $0<\varepsilon<1$,
$$
\Deltar(Z(\bff_\kappa),\varepsilon)\le c_{1} \,
\frac{\sum_{i=1}^{n}\log\|f_{\kappa,i}\|_{\sup}}{\varepsilon \kappa}.
$$ 
\end{proposition}

\begin{proof}
  This follows easily from Theorem \ref{thm:1} and Proposition \ref{prop:3}\eqref{item:8}.
\end{proof}

\begin{proof}[Proof of Theorem \ref{cor:2}]
  Following the notation in the statement of Theorem \ref{cor:2},
  $\nu_{Z(\bff_{\kappa})}$ is the discrete measure associated to
  $Z(\bff_{\kappa})$ and $\nu_{\haar}$ is the measure
  on~$(\C^{\times})^{n}$ induced by the Haar probability measure on
  $(S^{1})^{n}$.
 
We have to show that, for every continuous function with compact
support $h$, 
\begin{displaymath}
  \lim_{\kappa\to\infty}\int_{(\C^{\times})^{n}}h \, \d\nu_{Z(\bff_{\kappa})}=\int_{(\C^{\times})^{n}}h \, \d\nu_{\haar}.
\end{displaymath}
It is enough to prove the statement for the characteristic function $h_{U}$
of the open sets of the form 
\begin{equation} \label{eq:13}
U=  \{(z_{1},\dots, z_{n})\in(\C^\times)^n  \mid
r_{1,j}<|z_j|<r_{2,j},\,\alpha_j<\arg(z_j)<\beta_j \text{ for all }
j\},
\end{equation}
with $0\leq r_{1,j}<r_{2,j}\le \infty$, $ r_{i,j}\ne 1$ and
$-\pi<\alpha_j<\beta_j\leq\pi$, since any continuous function with
compact support can be uniformly approximated by a linear combinations
of the aforementioned characteristic functions.

Consider first the case where $U\cap(S^1)^n=\emptyset.$ Due to the
conditions imposed on the numbers $r_{i,j},$ there exists
$\varepsilon>0$ such that $U$ is disjoint with the set
\begin{displaymath}
 \{\bfxi\in\C^n
\mid \,1-\varepsilon<|\xi_j|<(1-\varepsilon)^{-1} \text{ for all } j\}.
\end{displaymath}
Hence,
$$
\int_{(\C^\times)^n}h_U\,\d\delta_{Z(\bff_\kappa)}=\frac{\deg(Z(\bff_\kappa)|_{U})}{\kappa^{n}\MV(\bfQ) }\le
\Deltar(\bff_{\kappa}, {\varepsilon}),  
$$
where $\MV(\bfQ)$ denotes the mixed volume of the polytopes $Q_{1},\dots, Q_{n}\subset\R^{n}$  and
$Z(\bff_\kappa)|_{U}=\sum_{\bfxi\in|Z(\bff_\kappa)|\cap U}m_\bfxi[\bfxi]$. Proposition \ref{prop:7}  implies
that this integral goes to $0$ for $\kappa\to\infty$, which proves the statement in this case, since 
$\int_{\C^n}h_U\,\d\nu_{\haar}=0$.

Consider now the case where $U\cap(S^1)^n\neq\emptyset$. 
Set 
\begin{math}
\ov U= \{\bfz \mid \alpha_j\le \arg(z_j)\le \beta_j \text{ for all }
j\}.   
\end{math}
Then
\begin{multline*}
  \int_{(\C^\times)^n}h_U\,\d\delta_{Z(\bff_\kappa)}-\int_{(\C^\times)^n}h_U\,\d\nu_{\haar}=
  \int_{(\C^\times)^n}h_U\,\d\delta_{Z(\bff_\kappa)}- \prod_{j=1}^{n} \frac{\beta_{j}-\alpha_{j}}{2\pi}\\
  =\int_{(\C^\times)^n}\bigg(h_{\overline{U}}-\prod_{j=1}^{n}
  \frac{\beta_{j}-\alpha_{j}}{2\pi}\bigg)\d\delta_{Z(\bff_\kappa)}-\int_{(\C^\times)^n}h_{\overline{U}\setminus
    U}\,\d\delta_{Z(\bff_\kappa)}.
\end{multline*}
We have
$$
\int_{(\C^\times)^n}\bigg|h_{\overline{U}}-\prod_{j=1}^{n} \frac{\beta_{j}-\alpha_{j}}{2\pi}\bigg|\d\delta_{Z(\bff_\kappa)}\leq   
  \bigg|\frac{\deg(Z(\bff_\kappa)_{\bfalpha,\bfbeta})}{\kappa^{n}\MV(\bfQ)} - \prod_{j=1}^{n}
\frac{\beta_{j}-\alpha_{j}}{2\pi} \bigg| \leq \Deltaa(\bff_\kappa).
$$
Again, Proposition \ref{prop:7}  implies
that this integral goes to $0$ for $\kappa\to\infty$.
On the other hand,  $\ov
U\setminus U$ is a union of a finite number of subsets $U_{l}$ of the
form \eqref{eq:13} such that $U_{l}\cap(S^1)^n=\emptyset$ for all $l$.
By the previous considerations, $\int_{(\C^\times)^n}h_{U_l}\d\delta_{Z(\bff_\kappa)}\to_{\kappa} 0$ and so
$\int_{(\C^\times)^n}h_{\ov U\setminus U}\, \d\delta_{Z(\bff_\kappa)}\to_{\kappa}0.$ 
Hence
\begin{displaymath}
\lim_{\kappa\to\infty}\int_{(\C^\times)^n}h_{{U}}\, 
\d\delta_{Z(\bff_\kappa)} = \prod_{j=1}^{n}
\frac{\beta_{j}-\alpha_{j}}{2\pi} =\int_{\C^n}h_U\,\d\nu_{\haar}=0,
\end{displaymath}
which concludes the proof.
\end{proof}


  

We will now consider random systems of Laurent polynomials with complex coefficients.
To explain and prove our results, we have to consider metrics and
measures on projective spaces over $\C$.  Let $N\ge 1$
and consider the standard Riemannian structure on $\C^{N}$ induced by
the Euclidean norm $\|\cdot\|_{2}$. Let $S^{2N-1}=\{\bfz\in \C^{N}
\mid \|\bfz\|_{2} =1\}$ be the unit sphere with the induced Riemannian
structure.  The map $S^{2N-1}\to \P(\C^{N})$ given by $(z_{0},\dots,
z_{N-1})\mapsto (z_{0}:\dots:z_{N-1})$ gives a principal bundle with
fiber $S^{1}$. The Fubini-Study metric on $\P(\C^{N})$ is defined as
the unique Riemannian structure such that this map is a Riemannian
submersion, see  \cite{KN:fdgII} for details.

The geodesics of $\P(\C^{N})$ coincide with lines. Hence, we can
define a distance between two points $\bfz_{1}$ and $\bfz_{2}$ as the
length of the line segment joining them, and we will denote it by
$\dist_{\FS}(\bfz_{1},\bfz_{2})$.  However, it will be more convenient
to consider the distance function $\dist:=\sin(\dist_{\FS})$.  This
 function can be computed with the formula
\begin{equation} \label{eq:14}
  \dist(\bfz_{1},\bfz_{2})= \sqrt{1-\Big
(\frac{|\langle
      \wt \bfz_{1},\wt \bfz_{2}\rangle|}{\|\wt \bfz_{1}\|_{2}\|\wt \bfz_{2}\|_{2}}\Big)^{2}}
\end{equation}
for any choice of representatives $\wt \bfz_{i}\in\C^{N}\setminus \{\bfzero\}$, $i=1,2$. 

\begin{lemma} \label{lemm:3} Let $\bfz_{1},\bfz_{2}\in \P(\C^{N})$ with $\wt \bfz_{i}\in \C^{N}$, $i=1,2$,  representatives of these
  points such that $\|\wt\bfz_{2}\|_2=1$. Then $\dist(\bfz_{1},\bfz_{2})\le
  \|\wt\bfz_{2}-\wt\bfz_{1}\|_{2}$.
\end{lemma}

\begin{proof}
We have that 
\begin{displaymath}
  \|\wt\bfz_{2}-\wt\bfz_{1}\|^{2}_{2}= \langle
      \wt\bfz_{2}-\wt\bfz_{1},\wt\bfz_{2}-\wt\bfz_{1}\rangle=1+\|\wt\bfz_{1}\|_{2}^{2}-2\operatorname{Re} (\langle
      \wt \bfz_{1},\wt \bfz_{2}\rangle)\ge 1+\|\wt\bfz_{1}\|_{2}^{2}-2|\langle
      \wt \bfz_{1},\wt \bfz_{2}\rangle|. 
\end{displaymath}
On the other hand, the
      formula   \eqref{eq:14} gives $  \dist(\bfz_{1},\bfz_{2})^{2}=  1-\big(\frac{|\langle
        \wt \bfz_{1},\wt
        \bfz_{2}\rangle|}{\|\wt\bfz_{1}\|_{2}}\big)^{2}$. Hence, 
\begin{displaymath}
  \|\wt\bfz_{2}-\wt\bfz_{1}\|^{2}_{2}-
  \dist(\bfz_{1},\bfz_{2})^{2}= \Big( \|\wt\bfz_{1}\|_{2}-\frac{|\langle
        \wt \bfz_{1},\wt
        \bfz_{2}\rangle|}{\|\wt\bfz_{1}\|_{2}}\Big)^{2}\ge0, 
\end{displaymath}
which proves the statement.
\end{proof}

We will need the following  \L ojasiewicz inequality for a 
hypersurface of a complex projective space. For a homogeneous
polynomial $f\in \C[x_{0},\dots, x_{N-1}]$ of degree $d,$ and a point 
$\bfz\in \P(\C^{N})$, the value 
\begin{displaymath}
  \frac{|f(\bfz)|}{\|\bfz\|_{2}^{d}}
\end{displaymath}
is well-defined. For a subset $E\subset \P(\C^{N})$, we write
$\dist(\bfz, E)$ for the distance between~$\bfz$ and $E$.

\begin{lemma} \label{lemm:2}
  Let  $f\in \C[x_{0},\dots,x_{N-1}]$ be a homogeneous polynomial of degree $d\ge 0$ and
  $\bfz\in \P(\C^{N})$. Then 
  \begin{displaymath}
    \frac{|f(\bfz)|}{\|\bfz\|_{2}^{d}}\ge  \bigg(\sup_{\bfx\in \P(\C^{N})}
   \frac{|f(\bfx)|}{\|\bfx\|_{2}^{d}}   \bigg) \dist(\bfz,V(f))^{d}. 
  \end{displaymath}
\end{lemma}

\begin{proof}
  Let $\bfz, \bfx\in \P(\C^{N})$ such that $\bfx\notin V(f)$.
  Let $\wt\bfz, \wt\bfx$ be representatives 
  of these points in the sphere $S^{2N-1}$ and set 
\begin{math}
  f_{\wt \bfx}(t)=f(\wt \bfz+ t \wt \bfx) \in \C[t].
\end{math}
This is a univariate polynomial of degree $d$ with leading coefficient
$f(\wt\bfx)$. Then, there exist $\xi_{j}\in \C$, $j=1,\dots, d$, such
that
\begin{math}
 f_{\wt
  \bfx}=f(\wt\bfx)\prod_{j}(t-\xi_{j}) 
\end{math}
and so 
\begin{displaymath}
  |f(\wt\bfz)|= |f_{\wt \bfx}(0)|=|f(\wt\bfx)| \prod_{j}|\xi_{j}| .
\end{displaymath}
For each $j$, we have that $\wt \bfz + \xi_{j}\wt\bfx\in V(f)$. Using Lemma \ref{lemm:3}, we deduce that
\begin{displaymath}
  |\xi_{j}|=\|(\wt \bfz + \xi_{j}\wt\bfx)-\wt\bfz\|_{2} \ge \dist(\bfz, \wt \bfz + \xi_{j}\wt\bfx) \ge  \dist(\bfz, V(f)).
\end{displaymath}
We deduce that
\begin{displaymath}
      \frac{|f(\bfz)|}{\|\bfz\|_{2}^{d}}= |f(\wt\bfz)|\ge |f(\wt\bfx)|   \dist(\bfz, V(f))^{d}
=\frac{|f(\bfx)|}{\|\bfx\|_{2}^{d}}  \dist(\bfz, V(f))^{d}.
\end{displaymath}
Since this holds for all $\bfx\notin V(f)$, the result follows. 
\end{proof}

Let $\mu_{\FS}$ denote the measure on $\P(\C^{N})$ induced by the
Fubini-Study metric. Then $\mu_{\FS}(\P(\C^{N}))= \frac{\pi^{N-1}}{(N-1)!}$. We
will consider the normalized measure given by
\begin{displaymath}
\mu= \frac{(N-1)!}{\pi^{N-1}} \mu_{\FS}. 
\end{displaymath}

A result of Beltr\'an and Pardo \cite[Theorem 1]{BP:edcnsm} shows
that, for a hypersurface $H\subset \P(\C^{N})$ of degree $d$, 
the normalized measure of the tube
around $H$ of radius $\rho\ge0$ is bounded from above by 
\begin{equation*}
15 d (N-1)^{2}\rho^{2}.  
\end{equation*}
Applying this result, we deduce the following bound for the volume of
the set where a polynomial can take small values. For  $\delta > 0$ and a homogeneous
polynomial $f\in \C[\bfx]$, we consider the
subset of $\P(\C^{N})$ given by
 \begin{equation} \label{eq:8}
V(f)_{\delta} =  \Big\{\bfz \in \P(\C^{N}) \Big|\
    \frac{|f(\bfz)|}{\|\bfz\|_2^{d}}<\delta \Big\}.
 \end{equation}

 \begin{proposition} \label{prop:4} Let $\delta > 0$ and $f\in
   \C[x_{0},\dots,x_{N-1}]$ a homogeneous polynomial of degree $d\ge
   1$. Then
  \begin{displaymath}
    \mu(V(f)_{\delta}) \le 15
    d N^{3} \Big( \frac{\delta}{\|f\|_{\sup}}\Big) ^{\frac2d}.
  \end{displaymath}
  In particular, if $f\in \Z[x_{0},\dots,x_{N-1}]$, then $
  \mu(V(f)_{\delta} ) \le 15 d N^{3}
  {\delta}^{\frac2d}$.
\end{proposition}

\begin{proof}
Let $\bfz\in V(f)_{\delta}$. Using Lemma \ref{lemm:2}, we deduce that
\begin{displaymath} 
\delta> \bigg(\sup_{\bfx}
\frac{|f(\bfx)|}{\|\bfx\|_{2}^{d}} \bigg) \dist(\bfz,V(f))^{d}
\ge \frac{\|f\|_{\sup}}{N^{\frac{d}{2}}} \dist(\bfz,V(f))^{d}.
\end{displaymath}
Hence, 
\begin{displaymath} 
\dist(\bfz,V(f))<
N^{\frac{1}{2}}\bigg(\frac{\delta}{\|f\|_{\sup}}
\bigg)^{\frac1d}
\end{displaymath}
and so $V(f)_{\delta}$ is contained in the tube around $V(f)$ of
radius $N^{\frac{1}{2}}\big(\frac{\delta}{\|f\|_{\sup}}
\big)^{\frac1d}$. The first part of the result follows then from the
Beltr\'an--Pardo bound for the volume of this tube.  The second part
follows from the fact that $ \|f\|_{\sup}\ge |f|\ge 1$, because of the
inequality \eqref{eq:17} and the fact
that the coefficients of $f$ are integer numbers.
\end{proof}

Let us keep the preceding notation and set  $\bfcA_{\kappa}=(\cA_{\kappa,1},\dots,
\cA_{\kappa,n})$ with $\cA_{\kappa,i}$ as in~\eqref{eq:4}.
Each point of the
projective space $\P(\C^{\bfcA_{\kappa}})$ can be identified with a system
$\bff_\kappa=(f_{\kappa,1},\dots,f_{\kappa,n})$ of Laurent polynomials such that
$\supp(f_{\kappa,i})\subset \kappa Q_{i}$, $i=1,\dots, n$, modulo a multiplicative
scalar. The associated cycle $Z(\bff_\kappa)$ is well-defined,
since it does not depend on this multiplicative scalar.

Set $\mu_{\kappa}$ for the normalized Fubini-Study measure on $\P(\C^{\bfcA_{\kappa}})$ and
let $\lambda_{\kappa}\colon  \P(\C^{\bfcA_{\kappa}}) \to \R_{\ge0}$ be a probability density
function, that is, a $\mu_{\kappa}$-measurable function with
\begin{displaymath}
\int_{\P(\C^{\bfcA_{\kappa}})}\lambda_{\kappa}\, \d\mu_{\kappa}=1. 
\end{displaymath}
Let $\bff_\kappa$ be a random system of Laurent polynomials with
$\supp(f_{\kappa,i})\subset \kappa Q_{i}$, $i=1,\dots, n$, distributed according
to the probability law given by $\lambda_{\kappa}$ with respect to $\mu_{\kappa}$. We
can then consider the angle discrepancy of $Z(\bff_\kappa)$ and, for
$0<\varepsilon< 1$, the radius discrepancy of $Z(\bff_\kappa)$ with respect to
$\varepsilon$, as random variables on $\P(\C^{\bfcA_{\kappa}})$. We denote
by $ \E(\Deltaa(Z(\bff_\kappa));{\lambda_{\kappa}})$ and
$\E(\Deltar(Z(\bff_\kappa),\varepsilon);{\lambda_{\kappa}})$ the expected value of
these random variables.

\begin{theorem}\label{eqprob0}
  For $\kappa\ge1$, let $\lambda_{\kappa}$ be a probability density function
  on $\P(\C^{\bfcA_{\kappa}})$ and
  $\bff_\kappa=(f_{\kappa,1},\dots,f_{\kappa,n})$ a random system of
  Laurent polynomials with $\supp(f_{\kappa,i})\subset \kappa Q_{i}$,
  $i=1,\dots, n$, distributed according to the probability law given
  by $\lambda_{\kappa}$ with respect to~$\mu_{\kappa}$. Assume that the
  sequence $(\lambda_{\kappa})_{\kappa\ge1}$ is uniformly bounded. Then
  there is a constant $c_{2} >0$ which does not depend on $\kappa$
  such that
$$
\E(\Deltaa(Z(\bff_\kappa));{\lambda_{\kappa}})\le c_{2} \, \frac{\log(\kappa+1)^{\frac{2}{3}n-\frac13}}{\kappa^{\frac13}}
$$
and, for any $0<\varepsilon<1$,
$$
\E(\Deltar(Z(\bff_\kappa),\varepsilon);{\lambda_{\kappa}})\le  c_{2} \,
\frac{\log(\kappa+1)}{\varepsilon \kappa}.
$$ 
In particular, 
\begin{displaymath}
  \lim_{\kappa\to\infty}\E(\Deltaa(Z(\bff_\kappa));{\lambda_{\kappa}})=0
  \quad\text{and}\quad \lim_{\kappa\to\infty}\E(\Deltar(Z(\bff_\kappa),\varepsilon);{\lambda_{\kappa}})=0.
\end{displaymath}
\end{theorem}

\begin{proof}
  We first estimate the expected value of the angle discrepancy,
  which is given by the formula
  \begin{displaymath} 
\E(\Deltaa(Z(\bff_\kappa));\lambda_{\kappa})=\int_{\P(\C^{\bfcA_\kappa})}\Deltaa(Z(\bff_\kappa))\lambda_\kappa(\bff)
\, \d\mu_{\kappa}.
  \end{displaymath}

  Consider the Minkowski sum $Q=\sum_{i=1}^{n}Q_{i}$, which is a
  lattice polytope on $\R^{n}$ of dimension $n$ because of the
  assumption that the mixed volume of $Q_{1},\dots,Q_{n}$ is positive.
  For each primitive vector $\bfv\in \Z^{n}$ which is an inner normal
  to a facet of~$Q$, consider the directional resultant
\begin{displaymath}
R_{\kappa,\bfv}=\Res_{\cA_{\kappa,1}^{\bfv},\dots,\cA_{\kappa,n}^{\bfv}}\in  \Z[\bfu_{1},\dots, \bfu_{n}],  
\end{displaymath}
where $\bfu_{i}$ is a group of $\#\cA_{\kappa,i}$
variables. Proposition \ref{prop:3}\eqref{item:7} implies that its
total degree is bounded by $\deg(R_{\kappa,\bfv})=
c_{\bfv}\kappa^{n-1}$ for a constant $c_{\bfv}$ independent of
$\kappa$. Its total number of variables is $\# \bfcA_{\kappa} =
\sum_{i}\cA_{\kappa,i}= \sum_{i}\kappa Q_{i}\cap\Z^{n}$. This number
can be bounded by $c_{3}\kappa^{n}$ for a constant $c_{3}$ independent
of~$\kappa$.

Set $\delta_{\kappa,\bfv}=\kappa^{-2n \deg(R_{\kappa,\bfv})}$. Consider the
subset $ V(R_{\kappa,\bfv})_{\delta_{\kappa,\bfv}}\subset \P(\C^{\bfcA_{\kappa}}) $ as
defined in~\eqref{eq:8} and put
\begin{displaymath}
  U_{\kappa}=\bigcup_{\bfv}V(R_{\kappa,\bfv})_{\delta_{\kappa,\bfv}},
\end{displaymath}
the union being over all primitive inner normal vectors to facets
of $Q$. Using the fact that $0\leq \Deltaa(Z(\bff_\kappa))\leq1 $ and the hypothesis that
the functions $\lambda_{\kappa}$ are uniformly bounded, we deduce that
\begin{displaymath}
0\le \int_{U_{\kappa}}\Deltaa(Z(\bff_\kappa))\lambda_\kappa(\bff)
\, \d\mu_{\kappa}  \le  \Big( \sup_{\bff_\kappa} \lambda_{\kappa}(\bff_\kappa) \Big)
\mu_{d}(U_{\kappa})\le c_{4}  \mu_{\kappa}(U_{\kappa})
\end{displaymath}
for a constant $c_{4}$ independent of $\kappa$. By Proposition~\ref{prop:4}, 
\begin{multline} \label{eq:5}
  \mu_{\kappa}(U_{\kappa}) \le\sum_{\bfv}
  \mu_{\kappa}(V(R_{\kappa,\bfv})_{\delta_{\kappa,\bfv}}) \le \sum_{\bfv }15 \deg(R_{\kappa,\bfv}) \big(\#
  \bfcA_{\kappa}\big)^{3} \delta_{\kappa,\bfv}^{\frac{2}{\deg(R_{\kappa,\bfv})}}\\
 \le 15 \Big(\sum_{\bfv}c_{\bfv}  \kappa^{n-1}\Big) (c_{3} \kappa^{n})^{3}
 \kappa^{-4n} = c_{5} \kappa^{-1},
\end{multline}
with $c_{5} = 15 \big(\sum_{\bfv}c_{\bfv} \big) c_{3}^{3}$.  Hence,
 $\mu_\kappa(U_{\kappa})\to_{\kappa}0$ and so
$\int_{U_{\kappa}}\Deltaa(Z(\bff_\kappa))\lambda_\kappa(\bff_\kappa) \, \d\mu_{\kappa}\to_{\kappa}0 $ as
well.

Let $\bff_\kappa\in \P(\C^{\bfcA_{\kappa}})\setminus U_{\kappa}$ and choose a
representative $\wt \bff_\kappa=(\wt f_{\kappa,1},\dots, \wt f_{\kappa,n})\in
\C^{\bfcA_{\kappa}}\setminus \{\bfzero\}$ with $\|\wt \bff_\kappa\|_{2}=1$. By
Proposition \ref{prop:3}\eqref{item:7}, $\eta(\bff_\kappa)=\eta(\wt
\bff_\kappa)$. Note that the Minkowski sum $\sum_{i}\kappa Q_{i}$ coincides with
$\kappa Q$. Hence, the only non trivial directional resultants of the family
of finite sets $\cA_{\kappa,1},\dots, \cA_{\kappa,n}$ are those of the form
$R_{\kappa, \bfv}$ as considered above. As before, let $\Delta^{n}\subset\R^{n}$
be the standard $n$-simplex. Choose $e\ge1$ and $\bfb_{i}\in
\Z^{n}$ such that $Q_{i}\subset e\Delta^{n}+\bfb_{i}$ for all
$i$. Hence, $\kappa\cA_{\kappa,i}\subset \kappa e Q_{i}+\kappa\bfb_{i}$ for
all~$i$. Proposition~\ref{prop:3}\eqref{item:8} implies that there is
a constant $c_{6}$ independent of $\kappa$ such that
\begin{align*}
  \eta(\bff_\kappa)& \le \frac{1}{\kappa^{n}\MV(\bfQ)}\Big((\kappa e)^{n-1}
(n+\sqrt{n})\sum_{i=1}^{n}\log \|\wt f_{\kappa,i}\|_{\sup} + \frac12 \sum_{\bfv}
\|\bfv\|_{2} \log^{+}|R_{\kappa,\bfv}(\bff^{\bfv}_\kappa)^{-1}|\Big)\\
& \le \frac{1}{\kappa^{n}\MV(\bfQ)}\Big((\kappa e)^{n-1}
(n+\sqrt{n})\sum_{i=1}^{n}\log (\#\cA_{\kappa,i}) + n  \sum_{\bfv}
\|\bfv\|_{2} \deg(R_{\kappa,\bfv})\log(\kappa) \Big)\\
 & \le c_{6}\frac{\log(\kappa+1)}{\kappa},
\end{align*}
the second and fourth sums being over the primitive inner normals
$\bfv$ to the facets of~$Q$. Here, we used the fact that $\|\wt
f_{\kappa, i}\|_{\sup}\le \#\cA_{\kappa,i}$ for $\wt \bff_\kappa$ in the unit sphere of
$\C^{\bfcA_{\kappa}}$, the definition of the set $U_{\kappa}$, the bound $
\#\cA_{\kappa,i}\le \#\bfcA_{\kappa}\le c_{3}\kappa^{n}$ and the inequality
$\deg(R_{\kappa,\bfv})\le c_{\bfv}\kappa^{n-1}$ that we explained before.

Using Theorem \ref{thm:1} and the fact that the function
$t^\frac13\log\left(\frac{\alpha}{t}\right)^{\frac{n-1}{3}}$ is
increasing for small values of $t>0,$ we deduce that, for $\bff_\kappa\in
\P(\C^{\bfcA_{\kappa}})\setminus U_{\kappa}$,
\begin{multline*}
\Deltaa(Z(\bff_\kappa))\leq c_{7} \eta(\bff_\kappa)^{\frac13}\log\Big(\frac{c_{8}}{\eta(\bff_\kappa)}\Big)^{\frac{2}{3}(n-1)}
\\ \leq c_{9} \Big(\frac{\log(\kappa+1)}{\kappa}\Big)^{\frac13} \log \Big(\frac{\kappa}{\log(\kappa+1)}\Big)^{\frac{2}{3}(n-1)}
\leq c_{10} \, \frac{\log(\kappa+1)^{\frac23 n-\frac13}}{\kappa^{\frac13}}
\end{multline*}
for suitable constants $c_7$, $c_{8}$, $c_{9}$ and $c_{10}$. This
proves the first part of the statement.  For the radius discrepancy,
we proceed in a similar way: given $\varepsilon>0$, we write
$\E(\Deltar(\bff_\kappa,\varepsilon);\lambda_{\kappa})$ as an integral, which we split
into two parts. We bound the first using that
$0\leq\Deltar(\bff_\kappa,\varepsilon)\leq1$ and the estimate~\eqref{eq:5},
while the second integral can be bounded by applying Theorem
\ref{thm:1}.
\end{proof}

\begin{proof}[Proof of Theorem \ref{thm:3}]
  The proof is similar to the one given for Theorem \ref{cor:2}. Write
  $\nu_{\kappa}=\frac{\E(Z(\bff_\kappa);{\lambda_{\kappa}})}{\kappa^{n}\MV(\bfQ)}$ for short,
  where $\E(Z(\bff_\kappa);{\lambda_{\kappa}})$ is the expected zero density measure
  of~$\bff_{\kappa}$ .
  To prove the statement, it is enough to show that, for all  subsets
  $U$ as  in~\eqref{eq:13},
\begin{displaymath}
\lim_{\kappa\to \infty}\nu_{\kappa}(U)= \nu_{\haar}(U\cap (S^{1})^{n})=
\begin{cases}
\prod_{j=1}^{n} \frac{\beta_{j}-\alpha_{j}}{2\pi} & \text{ if } U\cap
(S^{1})^{n}\ne\emptyset, \\
0 & \text{ otherwise.}
\end{cases}
\end{displaymath}

If $U\cap(S^1)^n=\emptyset,$ then there exists
$\varepsilon>0$ such that
\begin{displaymath}
\deg(Z(\bff_\kappa)|_{U})\le
\deg(Z(\bff_\kappa)) \Deltar(\bff_\kappa, {\varepsilon})\le
\kappa^{n}\MV(\bfQ) \Deltar(\bff_\kappa, {\varepsilon})  
\end{displaymath}
and so 
\begin{math}
\nu_{\kappa}(U)\le \E(\Deltar(\bff_\kappa,\varepsilon);\lambda_{\kappa}). 
\end{math}
Theorem \ref{eqprob0} then implies that
\begin{math}
 \lim_{\kappa\to \infty}\nu_{\kappa}(U)=
0= \nu_{\haar}(U). 
\end{math}

If $U\cap(S^1)^n\neq\emptyset$, we set 
\begin{math}
\ov U= \{\bfz \mid \alpha_j\le \arg(z_j)\le \beta_j \text{ for all }
j\},
\end{math}
and then we have
\begin{displaymath}
\nu_{\kappa}(U) - \prod_{j=1}^{n} \frac{\beta_{j}-\alpha_{j}}{2\pi}
= \Big( \nu_{\kappa}(\ov U) - \prod_{j=1}^{n}
\frac{\beta_{j}-\alpha_{j}}{2\pi}\Big) - \nu_{\kappa}(\ov U\setminus U). 
\end{displaymath}
Set $R_{\kappa}=\prod_{\bfv}R_{\kappa,\bfv}$ for the product of the
directional resultants of 
$\cA_{1},\dots,\cA_{n}$. Then
\begin{align*}
\Big| \nu_{\kappa}(\ov U) - \prod_{j=1}^{n}
\frac{\beta_{j}-\alpha_{j}}{2\pi} \Big| & =
\int_{\P(\C^{\bfcA_\kappa})\setminus V(R_{\kappa})}\Big|\frac{\deg(Z(\bff_\kappa)_{\bfalpha,\bfbeta})}{\kappa^{n}\MV(\bfQ)} - \prod_{j=1}^{n}
\frac{\beta_{j}-\alpha_{j}}{2\pi} \Big|  \lambda_{\kappa}(\bff_\kappa)\,\d\mu_\kappa \\ 
& \le 
\int_{\P(\C^{\bfcA_\kappa})}\Deltaa(\bff_\kappa) \lambda_{\kappa}(\bff_\kappa)\,\d\mu_\kappa.
\end{align*}
We have that $\ov
U\setminus U$ is a union of a finite number of subsets $U_{l}$ of the
form \eqref{eq:13} such that $U_{l}\cap(S^1)^n=\emptyset$ for all $l$.
By the previous considerations, $\lim_{\kappa\to\infty}\nu_{\kappa}(U_{l})=0$ and so
$\lim_{\kappa\to\infty}\nu_{d}(\ov U\setminus U)= 0$. Theorem \ref{eqprob0} then
implies that
\begin{displaymath}
  \lim_{\kappa\to \infty}\nu_{\kappa}(U)=   \lim_{\kappa\to \infty}\nu_{\kappa}(\ov U)=
\prod_{j=1}^{n} \frac{\beta_{j}-\alpha_{j}}{2\pi} = \nu_{\haar}
(U).
\end{displaymath}
\end{proof}

\begin{remark} \label{rem:3} It is not clear to us whether the upper
  bound in Proposition \ref{prop:4} for the volume of the set
  $V(f)_{\delta}$ is sharp or not. It would
  be interesting to clarify this point, as a qualititive improvement
  on this bound might enlarge the range of applicability of theorems
  \ref{eqprob0} and \ref{thm:3}.
\end{remark}

\begin{remark} \label{rem:1} In some situations, it might be
  interesting to consider probability distributions on the complex
  linear space $\C^{\bfcA_{\kappa}}$ rather than on
  $\P(\C^{\bfcA_{\kappa}})$. For a point $\bff_\kappa\in \P(\C^{\bfcA_{\kappa}})$, the
   associated cycle $Z(\bff_\kappa)$ does not depend on the choice of a
  representative in $\C^{\bfcA_{\kappa}}$ for this point and, \emph{a
    fortiori}, the same holds for the angle and radius discrepancies of
  $Z(\bff_\kappa)$. Hence, one might consider random variables on
  $\C^{\bfcA_{\kappa}}$ arising from this cycle as random variables on
  $\P(\C^{\bfcA_{\kappa}})$, by applying Federer's coarea formula (see for
  instance~\cite[Theorem 20]{BP:edcnsm}).

  In precise terms, the normal Jacobian of the map
  $\varpi\colon \C^{\bfcA_{\kappa}}\setminus \{\bfzero\}\to \P(\C^{\bfcA_{\kappa}})$
  with respect to the Euclidean structure on
  $\C^{\bfcA_{\kappa}}\setminus \{\bfzero\}$ and the Fubini-Study one on
  $\P(\C^{\bfcA_{\kappa}})$ is given, for $\bfg\in \C^{\bfcA_{\kappa}}\setminus
  \{\bfzero\}$, by
\begin{displaymath}
 \NJ_{\bfg}\varpi= \|\bfg\|^{-2N_{\kappa}}
\end{displaymath}
with $N_{\kappa}=\#\bfcA_{\kappa}-1$. Given a probability
density function $\lambda_{\kappa}\colon \C^{\bfcA_{\kappa}}\to \R$, one might derive a corresponding
probability density function on $\P(\C^{\bfcA_{\kappa}})$ by integrating along
the fibers of $\varpi$ as 
\begin{equation} \label{eq:9}
  \lambda_{\kappa}(\bff_\kappa)= \frac{\pi^{N_{\kappa}}}{N_{\kappa}!}
  \int_{\varpi^{-1}(\bff_\kappa)}{\Lambda_{\kappa}(\bfg)}{\|\bfg\|_{2}^{2N_{\kappa}}} \,
  \d \varpi^{-1}(\bff_\kappa), 
\end{equation}
where $ \d \varpi^{-1}(\bff_\kappa)$ is the volume form of the fiber
$\varpi^{-1}(\bff_\kappa)$. The probability distribution given by $\Lambda_{\kappa}$ of,
for instance, the angle discrepancy, can then be computed, for any
Borel subset $I\subset [0,1]$, as
\begin{displaymath}
  \Prob(\Deltaa(Z(\bff_\kappa)) \in I;\Lambda_{\kappa})=\int_{\Deltaa^{-1}(I)}
  \lambda_{\kappa}(\bff_\kappa)\, \d \mu_{\kappa},
\end{displaymath}
with $\Deltaa^{-1}(I)= \{\bff_\kappa\in \P(\C^{\bfcA_{\kappa}})\mid \,
\Deltaa(Z(\bff_\kappa))\in I\}$.  This is a consequence of the coarea
formula.
\end{remark}

\begin{example} \label{exm:1} Let $\bff_\kappa=(f_{\kappa,1},\dots, f_{\kappa,n})$ be
  a random system of Laurent polynomials with $\supp(f_{\kappa,i})\subset
  \kappa Q_{i}$ whose coefficients $\{f_{\kappa,i,\bfa}\}_{\bfa\in\cA_{\kappa,i}}$ are
  independent complex Gaussian random variables with mean 0 and
  variance 1.  This is a probability distribution on~$\C^{\bfcA_{\kappa}}$
  whose density function is defined, for $\bff_\kappa\in \C^{\bfcA_{\kappa}}$,
  as
\begin{displaymath} \Lambda_{\kappa}(\bff_\kappa)= \prod_{i=1}^{n}
\prod_{\bfa\in\cA_{\kappa,i}}\frac{1}{\pi}\e^{-|f_{\kappa,i,\bfa}|^{2}}=\frac{1}{\pi^{\#\bfcA_{\kappa}}}\e^{-\|\bff_\kappa\|_{2}^{2}}.
\end{displaymath}

The random cycle $Z(\bff_\kappa)$ might be described by a probability
distribution on $\P(\C^{\bfcA_{\kappa-}})$. The corresponding density function
is the constant function $\lambda_{\kappa}=1$. This can be seen by computing the
integral along the fibers \eqref{eq:9}, or simply by observing that
$\Lambda_{\kappa}$ is a function of the radius $\|\bff_\kappa\|_{2}$.

Theorem \ref{thm:3} implies then that the sequence of roots of
$\bff_\kappa$ converge weakly to the equidistribution on $(S^{1})^{n}$
when $\kappa\to \infty$. In this way, we recover a result of Bloom and
Shiffman \cite[Example 3.5]{BS07}.
\end{example}

\bibliographystyle{amsalpha}
\bibliography{biblio}

\newcommand{\noopsort}[1]{} \newcommand{\printfirst}[2]{#1}
  \newcommand{\singleletter}[1]{#1} \newcommand{\switchargs}[2]{#2#1}
  \def\cprime{$'$}
\providecommand{\bysame}{\leavevmode\hbox to3em{\hrulefill}\thinspace}
\providecommand{\MR}{\relax\ifhmode\unskip\space\fi MR }
\providecommand{\MRhref}[2]{%
  \href{http://www.ams.org/mathscinet-getitem?mr=#1}{#2}
}
\providecommand{\href}[2]{#2}
\begin{thebibliography}{GKZ94}

\bibitem[AM96]{AM96}
F.~Amoroso and M.~Mignotte, \emph{On the distribution of the roots of
  polynomials}, Ann. Inst. Fourier (Grenoble) \textbf{46} (1996), 1275--1291.

\bibitem[Ber75]{Ber75}
D.~N. Bernstein, \emph{The number of roots of a system of equations},
  Funkcional. Anal. i Prilozen. \textbf{9} (1975), 1--4, English translation:
  Functional Anal. Appl. {\bf 9} (1975), 183--185.

\bibitem[Bil97]{Bil97}
Y.~Bilu, \emph{Limit distribution of small points on algebraic tori}, Duke
  Math. J. \textbf{89} (1997), 465--476.

\bibitem[Bou70]{Bourbaki:ema}
N.~Bourbaki, \emph{{\'E}l\'ements de math\'ematique. {A}lg\`ebre. {C}hapitres 1
  \`a 3}, Hermann, 1970.

\bibitem[BP07]{BP:edcnsm}
C.~Beltr{\'a}n and L.~M. Pardo, \emph{Estimates on the distribution of the
  condition number of singular matrices}, Found. Comput. Math. \textbf{7}
  (2007), 87--134.

\bibitem[BS07]{BS07}
T.~Bloom and B.~Shiffman, \emph{Zeros of random polynomials on {${\mathbb
  C}^m$}}, Math. Res. Lett. \textbf{14} (2007), 469--479.

\bibitem[CLO05]{CLO05}
D.~A. Cox, J.~Little, and D.~O'Shea, \emph{Using algebraic geometry}, second
  ed., Grad. Texts in Math., vol. 185, Springer, 2005.

\bibitem[DS13]{DS:srmet}
C.~D'Andrea and M.~Sombra, \emph{A {P}oisson formula for the sparse resultant},
  e-print arXiv:1310.6617, 2013.

\bibitem[ET50]{ET50}
P.~Erd{\"o}s and P.~Tur{\'a}n, \emph{On the distribution of roots of
  polynomials}, Ann. of Math. (2) \textbf{51} (1950), 105--119.

\bibitem[FRL06]{FR06}
C.~Favre and J.~Rivera-Letelier, \emph{\'{E}quidistribution quantitative des
  points de petite hauteur sur la droite projective}, Math. Ann. \textbf{335}
  (2006), 311--361, {\it Corrigendum} in Math. Ann. {\bf 339} (2007), 799--801.

\bibitem[Ful98]{Fulton:IT}
W.~Fulton, \emph{Intersection theory}, second ed., Ergebnisse der Mathematik
  und ihrer Grenzgebiete. 3. Folge., vol.~2, Springer-Verlag, 1998.

\bibitem[Gan54]{Gan54}
T.~Ganelius, \emph{Sequences of analytic functions and their zeros}, Ark. Mat.
  \textbf{3} (1954), 1--50.

\bibitem[GKZ94]{GKZ94}
I.~M. Gelfand, M.~M. Kapranov, and A.~V. Zelevinsky, \emph{Discriminants,
  resultants, and multidimensional determinants}, Birkh\"auser, 1994.

\bibitem[HN08]{HN08}
C.~P. Hughes and A.~Nikeghbali, \emph{The zeros of random polynomials cluster
  uniformly near the unit circle}, Compos. Math. \textbf{144} (2008), 734--746.

\bibitem[Kho91]{Kho91}
A.~G. Khovanskii, \emph{Fewnomials}, Transl. Math. Monogr., vol.~88, Amer.
  Math. Soc., 1991.

\bibitem[KN69]{KN:fdgII}
S.~Kobayashi and K.~Nomizu, \emph{Foundations of differential geometry},
  Interscience Tracts in Pure and Applied Mathematics, No. 15, John Wiley \&
  Sons, 1969.

\bibitem[Pet05]{Pet05}
C.~Petsche, \emph{A quantitative version of {B}ilu's equidistribution theorem},
  Int. J. Number Theory \textbf{1} (2005), 281--291.

\bibitem[PS93]{PS93}
P.~Pedersen and B.~Sturmfels, \emph{Product formulas for resultants and {C}how
  forms}, Math. Z. \textbf{214} (1993), 377--396.

\bibitem[Som04]{Som04}
M.~Sombra, \emph{The height of the mixed sparse resultant}, Amer. J. Math.
  \textbf{126} (2004), 1253--1260.

\bibitem[SZ04]{SZ04}
B.~Shiffman and S.~Zelditch, \emph{Random polynomials with prescribed {N}ewton
  polytope}, J. Amer. Math. Soc. \textbf{17} (2004), 49--108.

\end{thebibliography}
\end{document}